\documentclass[12pt]{article}
\usepackage{mycommands}
\usepackage{colonequals}
\usepackage{tikz}
\title{The Maximal Rank Conjecture}

\begin{document}
\maketitle

\begin{abstract}
Let $C$ be a general curve of genus $g$, embedded in $\pp^r$
via a general linear series of degree $d$.
In this paper, we prove the \emph{Maximal Rank Conjecture},
which determines the Hilbert function of $C \subset \pp^r$.
\end{abstract}

\section{Introduction}

A central object of study in algebraic geometry in the past couple of centuries
has been algebraic curves in complex projective space.
(In this paper, we work exclusively over $\cc$.)
These can be described in two basic ways:
\begin{itemize}
\item Parametric Coordinates: We take an abstract curve $C$ of genus $g$,
pick a line bundle $\mathcal{L}$ on $C$ of degree $d$, and $r + 1$ linearly independent
sections of $H^0(\mathcal{L})$ whose span is basepoint free.
This gives rise to a degree $d$ map $C \to \pp^r$.
\item Cartesian Coordinates: We take a saturated homogeneous ideal $I \subset \cc[x_0, x_1, \ldots, x_r]$
of height $r - 1$. This gives rise to a curve $V(I) \subset \pp^r$.
\end{itemize}
A natural question is: \emph{How do these points of view relate to each other?}

\medskip

The \emph{Brill--Noether theorem},
proven by Griffiths and Harris \cite{bn}, Gieseker \cite{gp}, Kleiman and Laksov \cite{kl}, and others,
describes the space of parametric curves
with general source:
If $C$ is a general curve of genus $g$, 
it states that there exists a nondegenerate degree $d$ map $C \to \pp^r$ if and only if
the \emph{Brill--Noether number $\rho(d, g, r)$} is nonnegative, where
\[\rho(d, g, r) \colonequals (r + 1) d - rg - r(r + 1);\]
and moreover, in this case, there exists a unique component
of Kontsevich's space of stable maps $\bar{M}_g(\pp^r, d)$
that both
dominates the moduli space of curves
$\bar{M}_g$ and whose general member is nondegenerate,
whose relative dimension over $\bar{M}_g$ is
\begin{equation} \label{reldim}
\rho(d, g, r) + \dim \aut \pp^r = \rho(d, g, r) + r(r + 2).
\end{equation}
We term stable maps corresponding to points in this component
\emph{Brill--Noether curves (BN-curves)}. Such stable maps
that lie in only one component will be termed \emph{interior} BN-curves.

For $r \geq 3$ --- which we suppose for the remainder of this paper --- it
is known that a general BN-curve
is an embedding of a smooth curve \cite{bneh}, and so we may identify it with its image.

A specific instance of the above natural question is: \emph{What does the homogeneous ideal
(``Cartesian equations'')
of a general BN-curve look like, in terms of $d$, $g$, and $r$?}

As a first step, one might ask for the dimension of the graded pieces
of the ideal, i.e.\ for space of polynomials
of each degree~$k$ that vanish on $C$, which can be described as the
kernel of the restriction map
\[H^0(\oo_{\pp^r}(k)) \to H^0(\oo_C(k)).\]

The dimensions of these spaces are known: When $k \geq 3$,
the condition $\rho(d, g, r) \geq 0$ implies upon rearrangement that
$\deg \oo_C(k) = kd \geq \frac{rk}{r + 1} \cdot g + rk \geq 2g - 1$,
so $\oo_C(k)$ is nonspecial.
When $k = 2$, the result follows from the Gieseker--Petri theorem \cite{gp}, which states that the tensor product map
$H^0(\oo_C(1)) \otimes H^0(K_C(-1)) \to H^0(K_C)$
is injective: Indeed, restricting to a basepoint-free pencil in $H^0(\oo_C(1))$,
we obtain a map
$H^0(K_C(-1))^{\oplus 2} \to H^0(K_C)$
whose kernel is $H^0(K_C(-2))$; thus 
$H^0(K_C(-2)) = 0$, and so $\oo_C(2)$ is nonspecial.
When $k = 1$, one observes that for $d \geq g + r$, the line bundle $\oo_C(1)$ is general, and hence nonspecial.
Finally, $d \leq g + r$ implies
upon rearrangement that $\rho(d, g, r) + r(r + 2) > \rho(d, g, r + 1) + (r + 1)(r + 3) - (r + 2)$, so
the dimension estimate \eqref{reldim} implies that the general such $C$
cannot be a projection from $\pp^{r + 1}$,
and thus $\dim H^0(\oo_C(1)) = r + 1$.
Putting this together, one obtains:
\[\dim H^0(\oo_{\pp^r}(k)) = \binom{r + k}{k} \tand \dim H^0(\oo_C(k)) = \begin{cases}
r + 1 & \text{if $k = 1$ and $d \leq g + r$;} \\
kd + 1 - g & \text{otherwise.}
\end{cases}\]

The natural conjecture --- made originally by
Severi in 1915 \cite{severi} when the Brill--Noether theorem was still a conjecture --- is:

\begin{conj}[Maximal Rank Conjecture] \label{c:mrc}
If $C \subset \pp^r$ is a general BN-curve ($r \geq 3$), the restriction maps
\[H^0(\oo_{\pp^r}(k)) \to H^0(\oo_C(k))\]
are of maximal rank (i.e.\ either injective or surjective).

Or equivalently, the dimension of the space of polynomials
of degree $k$ which vanish on $C$ is given by
\[\begin{cases}
\binom{r + k}{k} - (kd + 1 - g) & \text{if $kd + 1 - g \leq \binom{r + k}{k}$ and $k \geq 2$;} \\
0 & \text{otherwise.}
\end{cases}\]
\end{conj}

Previously, many special cases of the maximal rank conjecture have been studied
--- including
the case of rational curves in $\pp^3$ by Hirschowitz \cite{mrat};
the cases of nonspecial curves (i.e.\ the case $d \geq g + r$) \cite{ballns}
and space curves \cite{ball} by Ballico and Ellia;
the case of quadrics (i.e.\ for $k = 2$)
independently by Ballico \cite{ball2}, and by Jensen and Payne \cite{jp}; and many others.
In this paper, we give the first proof in full generality:

\begin{thm} \label{thm:mrc}
Conjecture~\ref{c:mrc} (the Maximal Rank Conjecture) holds.
\end{thm}

The proof of this conjecture spans multiple papers, of which the present paper
is the final.
A detailed overview of the entire series of papers is given in the research announcement \cite{over},
which we highly encourage the reader to consult.
While this research announcement is purely expository
and thus not logically necessary for the present paper,
it includes a detailed
discussion of the key ideas that will introduced in the present paper and
the motivation behind them --- a discussion omitted from the present paper.

Our argument will be by induction on a stronger inductive hypothesis.
More generally, we say a subscheme $T \subset \pp^r$ \emph{satisfies maximal
rank for polynomials of degree $k$} if the restriction map
\[H^0(\oo_{\pp^r}(k)) \to H^0(\oo_T(k))\]
is of maximal rank. Since $H^1(\oo_{\pp^r}(k)) = 0$, the long exact
sequence in cohomology attached to the short exact sequence of sheaves
\[0 \to \ii_{T \subset \pp^r} (k) \to \oo_{\pp^r}(k) \to \oo_T(k) \to 0\]
implies that $T \subset \pp^r$ satisfies maximal rank for polynomials of degree $k$
if and only if
\[H^0(\ii_{T \subset \pp^r}(k)) = 0 \tor H^1(\ii_{T \subset \pp^r}(k)) = 0;\]
the vanishing of $H^0(\ii_{T \subset \pp^r}(k))$ being equivalent to the injectivity
of the restriction map, and the vanishing of $H^1(\ii_{T \subset \pp^r}(k))$ being equivalent
to the surjectivity.
In particular, the condition of satisfying maximal rank is open,
and can therefore be approached via degeneration.

Our stronger inductive hypothesis will be as follows.
(Note that taking $n = \epsilon = 0$
recovers the Maximal Rank Conjecture, so this proves Theorem~\ref{thm:mrc} as desired.)

\begin{thm} \label{main}
Fix an inclusion $\pp^r \subset \pp^{r + 1}$ (for $r \geq 3$), and let $k$ be a positive integer.
Let $C \subset \pp^r$ be a general BN-curve or a general degenerate rational curve
of degree $0 < d < r$.
Let $D_1, D_2, \ldots, D_n \subset \pp^{r + 1}$ be independently
general BN-curves, which are required to be nonspecial if $k = 2$ and $r \geq 4$.
Let $p_1, p_2, \ldots, p_\epsilon \in \pp^r$
be a general set of points.
Then any subset of
\[T \colonequals C \cup ((D_1 \cup D_2 \cup \cdots \cup D_n) \cap \pp^r) \cup \{p_1, p_2, \ldots, p_\epsilon\} \subset \pp^r\]
which contains $C$
satisfies maximal rank for polynomials of degree $k$.
\end{thm}

\paragraph{Notation:} In the proof of Theorem~\ref{main}, we write $d$ and $g$ for the
degree and genus of $C$, and $d_i$ and $g_i$ for the degrees and genera of the $D_i$.
We also define
\[h \colonequals \epsilon + \sum_{i = 1}^n d_i.\]

As explained above, our argument will be by induction on $r$ and $k$; we shall reduce
Theorem~\ref{main} for $(r, k)$ inductively to
Theorem~\ref{main} for $(r - 1, k)$ and $(r, k - 1)$.
But first, for fixed $(r, k)$, we may make the following reductions:
\begin{enumerate}
\item \label{T}
Applying the uniform position principle
and Lemma~2.5 of \cite{hyp}, it suffices to prove that $T$ satisfies
maximal rank in Theorem~\ref{main}.
(Theorem~\ref{main} is stated for any subset of $T$ containing 
$C$
only because that is a more convenient inductive hypothesis.)

\item \label{eeq0}
If $kd + 1 - g + h < \binom{r + k}{k}$, then $T$ satisfies maximal
rank for polynomials of degree $k$
if and only if $T \cup \{p\}$ does, where $p \in \pp^r$ is a general point.
We may therefore increase $\epsilon$ (and therefore $h$) until
$kd + 1 - g + h \geq \binom{r + k}{k}$.

Alternatively, in any case, we may remove general points until $\epsilon = 0$.

We therefore suppose
\begin{equation} \label{epsgeq4}
kd + 1 - g + h \geq \binom{r + k}{k} \quad \text{if} \quad r \geq 4,
\end{equation}
and that
\begin{equation} \label{eps3}
\epsilon = 0 \quad \text{if} \quad r = 3.
\end{equation}
\end{enumerate}

\paragraph{Organization of the remainder of the paper:} 
We begin, in Section~\ref{sec:degenerations}, by studying the key degenerations 
that we shall use in our inductive argument.
Here we leverage a special case of the result, originally due to Ballico and Ellia \cite{linesgeq4},
that the union
of a general rational curve and general disjoint lines satisfies maximal rank for polynomials of any degree;
however, to make this paper more self-contained, we give an independent proof
of the special case of this fact that we shall use in Appendix~\ref{app:lines}.

Then in Sections~\ref{sec:space} and~\ref{sec:quadrics},
we establish Theorem~\ref{main} for space curves ($r = 3$)
and quadrics ($k = 2$).
For space curves, we leverage the result, originally due to Ballico and Ellia \cite{ball},
that establishes Theorem~\ref{thm:mrc} in this case;
however, to make this paper more self-contained, we give an independent proof
of this result in Appendix~\ref{app:space}.
These cases will form the base cases for our larger inductive argument.

In Section~\ref{extp}
we study the ``easy'' cases of Theorem~\ref{main}, where the restriction map
is far from being an isomorphism (i.e.\ either the dimension of the source
is much larger than the dimension of the target, or vice versa).
Intuitively, these are the easier cases
since the codimension in the space of all linear maps of those with non-maximal rank
is highest when the dimensions are far apart.

Finally, in Section~\ref{sec:inductive},
we use the machinery of the preceding
sections to reduce Theorem~\ref{main}
to a computation involving the existence of integers satisfying
certain systems of inequalities.
This computation is taken care of in Appendices~\ref{app:inequalities},
\ref{app:comp}, and~\ref{app:covers}.

\subsection*{Acknowledgements}

The author would like to thank Joe Harris for
his guidance throughout this research,
as well as Atanas Atanasov, Edoardo Ballico, Brian Osserman, Sam Payne, Ravi Vakil, Isabel Vogt, David Yang, and other members of the Harvard and MIT mathematics departments,
for helpful conversations or comments on this manuscript.
The author would also like
to acknowledge the generous
support both of the Fannie and John Hertz Foundation,
and of the Department of Defense
(DoD) through the National  Defense Science and Engineering Graduate Fellowship (NDSEG) Program.

\section{Degenerations \label{sec:degenerations}}

In this section, we outline the degenerations we shall use in the proof
of Theorem~\ref{main}.

\begin{defi}
We say that $r$, $k$, $d$, $g$, $d'$, $g'$, $h$, and $h'$
\emph{satisfy $I(r, k, d, g, d', g', h, h')$} if
\begin{align}
(k - 1) d' + 1 - g' + h' &\geq \frac{k}{r + k} \cdot \binom{r + k}{k} = \binom{r + k - 1}{k - 1} \label{upstairs} \\
kd - g - (k - 1)d' + g' + h - h' &\geq \frac{r}{r + k} \cdot \binom{r + k}{k} = \binom{r + k - 1}{k}. \label{downstairs}
\end{align}
\end{defi}

\begin{prop} \label{hir}
Let $r \geq 4$ and $H \subset \pp^r$ be a hyperplane.
Suppose that Theorem~\ref{main}
holds for $(r - 1, k)$ and $(r, k - 1)$.
Assume there exists:
\begin{itemize}
\item A specialization of $C$ to an interior BN-curve $C^\circ = C' \cup C''$,
with $C''$ contained in $H$ and $C'$ transverse to $H$,
with $C' \cap C''$ general, and with $C'$ of degree $d'$ and genus $g'$
(by convention $\emptyset$ is of degree $0$ and genus $1$).
\item Specializations $\{p_1^\circ, p_2^\circ, \ldots, p_\epsilon^\circ\}$
of $\{p_1, p_2, \ldots, p_\epsilon\}$,
and $D_i^\circ$ of each $D_i$,
to disjoint sets of distinct points with
\[\#(\{p_1^\circ, p_2^\circ, \ldots, p_\epsilon^\circ\} \cap H) + \sum \# (D_i^\circ \cap H) = h - h'.\]
\end{itemize}
Assume also that for some deformation $\tilde{C'}$ of $C'$,
\begin{gather*}
A \colonequals \big(C^\circ \cup D_1^\circ \cup D_2^\circ \cup \cdots \cup D_n^\circ \cup \{p_1^\circ, \ldots, p_\epsilon^\circ\}\big) \cap H \subset H \quad \text{and} \\
B \colonequals \tilde{C'} \cup \big(\big(D_1^\circ \cup D_2^\circ \cup \cdots \cup D_n^\circ \cup \{p_1^\circ, \ldots, p_\epsilon^\circ\}\big) \cap (\pp^r \smallsetminus H)\big) \subset \pp^r
\end{gather*}
each either:
\begin{enumerate}
\item Satisfy the assumptions of Theorem~\ref{main},
with polynomials of degree $k$ for $A$, or with polynomials of degree $k - 1$ for $B$;
\item Are the union
of subsets of hyperplane sections of general BN-curves
which are nonspecial if $k = 2$ for $A$, or if $k = 3$ for $B$; or
\item Are the union of a general rational curve and at most $k - 1$ for $A$, or $k - 2$ for $B$,
independently general lines,
together with some number of independently general points.
\end{enumerate}
If $r$, $k$, $d$, $g$, $d'$, $g'$, $h$, and $h'$ satisfy
$I(r, k, d, g, d', g', h, h')$,
then Theorem~\ref{main} holds for $T$.
\end{prop}
\begin{proof}
Since $C' \cap C''$ is general in $H$, and $C'$ is transverse to $H$,
we may arrange for $\tilde{C'}$ to pass through $C' \cap C''$,
so that $\tilde{C'} \cup C''$ is a deformation of $C' \cup C''$.
Since $C' \cup C''$ is an interior BN-curve,
we conclude that $\tilde{C'} \cup C''$ is also a specialization of $C$.

Write $T^\circ = \tilde{C'} \cup C'' \cup ((D_1^\circ \cup D_2^\circ \cup \cdots \cup D_n^\circ) \cap \pp^r) \cup \{p_1^\circ, p_2^\circ, \ldots, p_\epsilon^\circ\}$.
Then we have an exact sequence
\[0 \to \ii_{B \subset \pp^r}(k - 1) \to \ii_{T^\circ \subset \pp^r}(k) \to \ii_{T^\circ \cap H \subset H}(k) \to 0.\]
Thus, to show $H^0(\ii_{T^\circ \subset \pp^r}(k)) = 0$,
it suffices to show $H^0(\ii_{B \subset \pp^r}(k - 1)) = H^0(\ii_{T^\circ \cap H \subset H}(k)) = 0$.
Since $T^\circ \cap H$ can be specialized to $A$, it suffices to show
$H^0(\ii_{B \subset \pp^r}(k - 1)) = H^0(\ii_{A\subset H}(k)) = 0$,
or equivalently that
the restriction maps
\[H^0(\oo_H(k)) \to H^0(\oo_A(k)) \tand H^0(\oo_{\pp^r}(k - 1)) \to H^0(\oo_B(k - 1))\]
are both injective.
The condition
$I(r, k, d, g, d', g', h, h')$
implies
\begin{align*}
\chi(\oo_B(k - 1)) = (k - 1)d' + 1 - g' + h' &\geq \frac{k}{r + k} \cdot \binom{r + k}{k} = \dim H^0(\oo_{\pp^r}(k - 1)) \\
\text{and} \quad \chi(\oo_A(k)) = kd - g - (k - 1)d' + g' + h - h' &\geq \frac{r}{r + k} \cdot \binom{r + k}{k} = \dim H^0(\oo_H(k)),
\end{align*}
and so
\begin{gather*}
\dim H^0(\oo_B(k - 1)) \geq \chi(\oo_B(k - 1)) \geq \dim H^0(\oo_{\pp^r}(k - 1)) \\
\dim H^0(\oo_A(k)) \geq \chi(\oo_A(k)) \geq \dim H^0(\oo_H(k)).
\end{gather*}

It thus remains to show that
$A \subset H$, respectively $B \subset \pp^r$,
satisfies maximal rank for polynomials of degree $k$, respectively of degree $k - 1$.
This holds by our inductive hypothesis for Theorem~\ref{main},
or Theorem~1.3 of~\cite{hyp} --- unless $A \subset H$, respectively $B \subset \pp^r$,
is the union of a general rational curve and $k - 1$ or fewer, respectively $k - 2$ or fewer,
independently general lines,
together with some number of independently general points.

Since the union of a subscheme satisfying maximal rank for polynomials
of some degree with a general point
still satisfies maximal rank for polynomials of that degree,
it thus remains to show that the union of a general rational curve and $k - 1$
or fewer independently general lines in $\pp^r$ satisfies maximal rank for polynomials of degree $k$.
But for $r = 3$, this holds by a result of Hartshorne and Hirschowitz
\cite{lines3}.
And for $r \geq 4$, this is a special case of a theorem of Ballico and Ellia \cite{linesgeq4}; however, this special case is much easier to prove than
the general case, so to make this paper more self-contained we give a proof
of this special case in Appendix~\ref{app:lines}.
\end{proof}

\begin{defi}
We say that $r$, $h$, and $h'$
\emph{satisfy $A(r, h, h')$} if
\[0 \leq h' \leq \frac{h}{r + 1}.\]
\end{defi}

\begin{prop} \label{prop:crack}
For some integer $h'$ satisfying $A(r, h, h')$,
there exist
specializations
$D_i^\circ$ of each $D_i$, with
\[(D_1^\circ \cup D_2^\circ \cup \cdots \cup D_n^\circ) \cap (\pp^r \smallsetminus H) \subset \pp^r\]
a general set of $h'$ points, and
\[D_1^\circ \cap H, D_2^\circ \cap H, \ldots, D_n^\circ \cap H \subset H\]
a set of hyperplane sections of independently general
BN-curves.
\end{prop}
\begin{proof}
Since $\sum \deg D_i \leq h$, this follows from combining Lemmas~5.3 and~6.1 of~\cite{hyp}.
\end{proof}

\begin{defi}
We say that $r$, $h$, and $h'$ \emph{satisfy $J(r, h, h')$} if
\[r + \left \lfloor \frac{h}{r + 1} \right \rfloor \leq h' \leq h,\]
\emph{satisfy $K(r, h, h')$} if
\[h \leq 2r + 1 \tand 0 \leq h' \leq h,\]
\emph{satisfy $L(r, h, h')$} if
\[2 \leq h \leq 3r + 2 \tand h' = 2,\]
\emph{satisfy $M(r, h, h')$} if
\[h \leq 3r + 2 \tand r + 2 \leq h' \leq h - \left\lfloor \frac{r}{2} \right\rfloor,\]
and \emph{satisfy $N(r, h, h')$} if
\[h \geq 0 \tand r + \left \lfloor \frac{h}{r + 1} \right \rfloor \leq h' \leq h - r \cdot \left\lfloor \frac{h}{2r + 2}\right\rfloor.\]
\end{defi}

\begin{prop} \label{prop:hyp}
There exist specializations $\{p_1^\circ, p_2^\circ, \ldots, p_\epsilon^\circ\}$
of $\{p_1, p_2, \ldots, p_\epsilon\}$,
and $D_i^\circ$ of each $D_i$,
with
\[\#(\{p_1^\circ, p_2^\circ, \ldots, p_\epsilon^\circ\} \cap H) + \sum \# (D_i^\circ \cap H) = h - h',\]
so that
\begin{gather*}
\{p_1^\circ, p_2^\circ, \ldots, p_\epsilon^\circ\} \cap H \subset H \quad \text{and} \quad \\
\{p_1^\circ, p_2^\circ, \ldots, p_\epsilon^\circ\} \cap \pp^r \smallsetminus H \subset \pp^r
\end{gather*}
are sets of general points, and
\begin{gather*}
D_1^\circ \cap H, D_2^\circ \cap H, \ldots, D_n^\circ \cap H \subset H \quad \text{and} \\
D_1^\circ \cap \pp^r \smallsetminus H, D_2^\circ \cap \pp^r \smallsetminus H, \ldots, D_n^\circ \cap \pp^r \smallsetminus H \subset \pp^r
\end{gather*}
are sets of subsets of hyperplane sections of independently general
BN-curves, provided that
$r$, $h$, and $h'$ satisfy $J(r, h, h')$.

Moreover, such specializations exist so that
the second of these sets can be specialized to a set of subsets
of hyperplane sections of independently general nonspecial BN-curves
if $r$, $h$, and $h'$ satisfy one of $K(r, h, h')$, $L(r, h, h')$, $M(r, h, h')$,
or $N(r, h, h')$.
\end{prop}
\begin{proof}
Note that:
\begin{itemize}
\item If $h \geq 1$, and $r$, $h$, and $h'$ satisfy $K(r, h, h')$, then they also satisfy
either $K(r, h - 1, h')$ or $K(r, h - 1, h' - 1)$.

\item If $h \geq 2r + 2$, and $r$, $h$, and $h'$ satisfy $J(r, h, h')$, then they also satisfy
either $J(r, h - 1, h')$ or $J(r, h - 1, h' - 1)$.
The same holds with $J$ replaced by $L$, $M$, or $N$.

\item If $h \leq 2r + 1$, and $r$, $h$, and $h'$ satisfy $J(r, h, h')$,
then they also satisfy $K(r, h, h')$.
The same holds with $J$ replaced by $L$, $M$, or $N$.
\end{itemize}
Since general points in $\pp^r$ can be specialized to general points
in either $H$ or $\pp^r$, we may therefore reduce $\epsilon$ until
$\epsilon = 0$.

By combining Lemmas~5.4 and~6.1 of~\cite{hyp}, we see
such a specialization exists if $r$, $h$, and $h'$ satisfy $J(r, h, h')$;
moreover, we can take
the second of these sets to be a set of subsets
of hyperplane sections of independently general nonspecial BN-curves
if $N(r, h, h')$ is satisfied.

It thus remains to consider the case when one of $K(r, h, h')$,
$L(r, h, h')$, or $M(r, h, h')$ is satisfied.
Note that
\[d_i \geq d_i - \rho(d_i, g_i, r + 1) = (r + 1) \cdot (g_i + r + 1 - d_i + 1).\]
In particular, $D_i$ can only be special if its degree is at least $2r + 2$.

If $K(r, h, h')$ is satisfied, we thus conclude all $D_i$ are nonspecial.
Similarly, if $L(r, h, h')$ or $M(r, h, h')$ is satisfied,
then since a BN-curve in $\pp^{r + 1}$ has degree at least $r + 1$,
we thus conclude that either all $D_i$ are nonspecial
or $n = 1$ and $g_1 = d_1 - r$.

Since the 
the desired result when all $D_i$ are nonspecial
follows from Corollary~4.2 of~\cite{hyp},
it remains to consider the case when $L(r, h, h')$ or $M(r, h, h')$ is satisfied,
$n = 1$, and $g_1 = d_1 - r$.

If $L(r, h, h')$ is satisfied, the result now follows from Lemma~5.3 of \cite{hyp}.
It thus remains to consider the case when $r$, $h$, and $h'$ satisfy $M(r, h, h')$ and not $N(r, h, h')$,
i.e.\ when
\[h \leq 3r + 2 \tand h - r + 1 \leq h' \leq h - \left\lfloor \frac{r}{2} \right \rfloor.\]

By the uniform position principle, the points of $D_1 \cap \pp^r$ are in
linear general position. We may therefore apply an automorphism of $\pp^r$
so that exactly $h - h' \leq r - 1$
of these points lie in $H$.

Since the automorphism group of $\pp^r$ acts transitively on sets of
$h - h' \leq r - 1$ points in linear general position,
$D_1 \cap H$ can be assumed to be general;
in particular, it is a subset of the hyperplane section
of a general rational normal curve.

It remains to see $D_1 \cap \pp^r \smallsetminus H$,
which is a subset of a hyperplane section of a general special BN-curve,
can be specialized to a subset
of a hyperplane section of a general nonspecial BN-curve.
For this, we apply Theorem~1.8 of \cite{rbn}
to degenerate $D_1 \hookrightarrow \pp^r$ to a stable map
$f \colon D \cup_\Gamma \pp^1 \to \pp^r$,
with $f|_D$ a general BN-curve of degree $d_1 - \lfloor \frac{r}{2}\rfloor$
and genus $d_1 - \lfloor \frac{r}{2}\rfloor - r - 1$
(which is in particular nonspecial),
and $f|_{\pp^1}$ of degree $\lfloor \frac{r}{2}\rfloor$,
and $\# \Gamma = \lfloor \frac{r}{2}\rfloor + 2$.
Since $h' \leq h - \lfloor \frac{r}{2}\rfloor$ by assumption,
we can arrange for $D_1 \cap \pp^r \smallsetminus H \subset f(D) \cap H$.
\end{proof}

\begin{defi}
We say that $r$, $d$, $g$, $d'$, $g'$, $n$, and $t$
\emph{satisfy $X(r, d, g, d', g', n, t)$} if
\begin{align}
-g' &\geq 0 \\
d' + g' - 1 &\geq 0 \\
t + g' &\geq 0 \\
g - g' - n + 1 &\geq 0 \label{a0} \\
r(d - d') - (r - 1)(g - g') + (r - 1)n - r^2 + 1 &\geq 0 \label{b0}\\
n + g' - 1 &\geq 0 \label{c0}\\
d' - n &\geq 0 \label{d0}\\
r(d - d') - (r - 4)(g - g') - 2n - 2r + 2 &\geq 0 \label{h0} \\
r + 2 - n - g' &\geq 0 \label{e0}\\
2n + d + g' - d' - g - r - 1 &\geq 0. \label{f0}
\end{align}
\end{defi}

\begin{prop} \label{prop-rat}
Let $r \geq 4$ and
$d$, $g$, $d'$, $g'$, and $t$ be integers which satisfy
\[(r + 1) d - rg - r(r + 1) \geq 0,\]
and suppose there exists an integer $n$ which satisfies
$X(r, d, g, d', g', n, t)$.
Then there exists an interior BN-curve $C' \cup C'' \subset \pp^r$
of degree $d$ and genus $g$,
with $C''$ a general BN-curve in a hyperplane $H$, and $C'$ a general
union of a rational curve of degree $d' + g'$ and $-g' \leq t$ general lines,
all transverse to $H$.
\end{prop}
\begin{proof}
Let $n$ be some such integer.
Write $g'' = g - g' - n + 1$ and $d'' = d - d'$.
Upon rearrangement, \eqref{a0}
gives $g'' \geq 0$, and \eqref{b0} gives
$\rho(d'', g'', r - 1) \geq 0$.
We may therefore let $C'' \subset H$ be a general BN-curve
of degree $d''$ and genus $g''$.
Note that $C''$ passes through $n$ general points in $H$,
by Theorem~1.2 of \cite{ibe} together with our assumption \eqref{h0}.
Since $d' \geq n$ by \eqref{d0},
and the hyperplane section of a general union of rational curves of total
degree $d'$
is a general set of $d'$ points by Corollary~1.5 of~\cite{tan},
we may thus let $C'$ be a general union of a rational curve of degree $d' + g'$
and $-g' \leq t$ general lines,
passing through $\Gamma$.
Since $n \geq 1 - g'$ by \eqref{c0}, we may further
suppose that every component of $C'$ meets $\Gamma$.

By construction, $C' \cup C''$ is of degree $d' + d'' = d$
and genus $g' + g'' + n - 1 = g$. It thus remains to see it is an interior BN-curve,
which holds by Corollary~1.11 of~\cite{rbn}.
\end{proof}

\begin{defi}
We say that $r$, $d$, $g$, $d'$, $g'$, and $n$
\emph{satisfy $Y(r, d, g, d', g', n)$} if
\begin{align}
g' &\geq 0 \\
(r + 1) d' - rg' - r^2 - r &\geq 0 \\
(2r - 3) d' - (r - 2)^2 g' - 2r^2 + 3r - 9 &\geq 0 \label{x} \\
g - g' - n + 1 &\geq 0 \label{z1} \\
r(d - d') - (r - 1)(g - g') + (r - 1) n - r^2 + 1 &\geq 0 \label{z2} \\
n - 1 &\geq 0 \label{z3} \\
d' - n &\geq 0 \label{y} \\
r (d - d') - (r - 4)(g - g') - 2n - 2r + 2 &\geq 0 \label{z4}\\
2n + d + g' - d' - g - r - 2 &\geq 0. \label{z5}
\end{align}
\end{defi}

\begin{prop} \label{prop2}
Let $r \geq 4$,
and $d$, $g$, $d'$, and $g'$ be integers
which satisfy
\[(r + 1) d - rg - r(r + 1) \geq 0,\]
and suppose there exists an integer $n$
satisfying $Y(r, d, g, d', g', n)$.
Then there exists an interior BN-curve $C' \cup C'' \subset \pp^r$
of degree $d$ and genus $g$, with $C''$ a general BN-curve
in a hyperplane, and $C'$ a general BN-curve of degree $d'$
and genus $g'$ transverse to $H$ ---
such that $C' \cap H$ is a set of $d'$ general points in $H$.
\end{prop}
\begin{proof}
Let $n$ be the minimal such integer.
Note that \eqref{x} and \eqref{y}, together with our assumption that $t \geq 0$,
imply
\begin{equation} \label{xp}
(2r - 3) d' - (r - 2)^2 (g' - d' + n) - 2r^2 + 3r - 9 \geq 0. \tag{\ref{x}$'$}
\end{equation}
Additionally, \eqref{y} implies
\begin{equation} \label{yp}
d' - n \geq 0. \tag{\ref{y}$'$}
\end{equation}

Since the left-hand sides of \eqref{x} and \eqref{y}
are nonincreasing in $n$, it follows that $n$ is also the minimal
integer satisfying the system of inequalities
\eqref{xp}, \eqref{z1}, \eqref{z2}, \eqref{z3}, \eqref{yp}, \eqref{z4}, \eqref{z5}.
Additionally, note that \eqref{x} 
also implies
\[(2r - 3) (d' + 1) - (r - 2)^2 g' - 2r^2 + 3r - 9 \geq 0.\]
We conclude the desired curve exists by
applying Theorem~1.2 and Remark~1.3 of \cite{rbn2}.
\end{proof}

\begin{lm} \label{lm:forprop3}
Let $X \subset \pp^r$ be a subscheme of codimension at least $2$,
and $\Delta \subset X$ be a set of $r + 2$ points which are general in some nondegenerate component
of the smooth locus of $X$. Then for every integer $m \geq 1$,
there exists a BN-curve $C$ of degree $rm$ and genus $(r + 1)(m - 1)$
whose intersection with $X$ is exactly the reduced scheme $\Delta$.
\end{lm}
\begin{proof}
We argue by induction on $m$.

When $m = 1$, we first
note that $\aut \pp^r$ acts transitively on sets of $r + 2$
points in linear general position. Applying an automorphism to a rational normal curve,
we may thus find a rational normal curve $C$ passing through $\Delta$.
It is a classical fact (and also an immediate consequence of Theorem~1.3 of~\cite{aly})
that $N_C \simeq \oo_{\pp^1}(r + 2)^{\oplus (r - 1)}$, and so
$N_C(-\Delta) \simeq \oo_{\pp^1}^{\oplus (r - 1)}$
has vanishing cohomology and is generated by global sections.
We may thus deform $C$ to a curve passing through $\Delta$
which avoids any (excess) intersection with any subvariety of codimension at least $2$.

For the inductive step, we let $C_0$ be a BN-curve of degree $r(m - 1)$ and genus $(r + 1)(m - 2)$
whose intersection with $X$ is exactly the reduced scheme $\Delta$.
We then let $\Delta_0$ be a set of $r + 2$ general points on $C_0$.
Applying our inductive hypothesis again, we may find a rational normal curve $C_1$
whose intersection with $X \cup C_0$ is exactly the reduced scheme $\Delta'$.
Taking $C = C_0 \cup C_1$ completes the proof, as this is a BN-curve by Theorem~1.6 of~\cite{rbn}.
\end{proof}

\begin{defi} We say that $r$, $d$, $g$, $d'$, $g'$, $n$, and $m$
\emph{satisfy $Z(r, d, g, d', g', n, m)$} if
\begin{align}
(r + 1) d' - rg' - r^2 - r &\geq 0 \\
g' - (r + 1) m &\geq 0 \\
(2r - 3) (d' - rm) - (r - 2)^2 (g' - (r + 1)m) - 2r^2 + 3r - 9 &\geq 0 \label{inter} \\
g - g' - n + 1 &\geq 0 \\
r(d - d') - (r - 1)(g - g') + (r - 1) n - r^2 + 1 &\geq 0 \\
n - 1 &\geq 0 \\
(d' - rm) - n &\geq 0 \\
r (d - d') - (r - 4)(g - g') - 2n - 2r + 2 &\geq 0 \\
2n + d + g' - d' - g - r - 2 &\geq 0 \label{interior} \\
2(d' - rm) - (r - 3)(g' - (r + 1)m - 1) - (r - 1)(r + 2) &\geq 0 \label{r2} \\
m &\geq 0.
\end{align}
\end{defi}

\begin{prop} \label{prop3}
Let $r \geq 4$,
and $d$, $g$, $d'$, and $g'$ be integers
which satisfy
\[(r + 1) d - rg - r(r + 1) \geq 0,\]
and suppose there exist integers $n$ and $m$ satisfying
$Z(r, d, g, d', g', n, m)$.
Then there exists an interior BN-curve ${C'_1} \cup {C'_2} \cup {C''} \subset \pp^r$
of degree $d$ and genus $g$, with ${C''}$ a general BN-curve
in a hyperplane, and ${C'_1} \cup {C'_2}$ a BN-curve of degree $d'$
and genus $g'$ transverse to $H$ ---
such that ${C'_2} \cap H$ is a set of general points in $H$,
and ${C'_1}$ is either a BN-curve
which is general independent from ${C'_2} \cap H$ or ${C'_1} = \emptyset$.
\end{prop}
\begin{proof}
By Proposition~\ref{prop2},
there exists a BN-curve ${C'_2} \cup {C''} \subset \pp^r$
of degree $d - rm$ and genus $g - (r + 1)m$, with ${C''}$ a general BN-curve
in a hyperplane $H$, and ${C'_2}$ a general BN-curve of degree $d' - rm$
and genus $g' - (r + 1)m$ transverse to $H$ ---
such that ${C'_2} \cap H$ is a set of $d' - rm$ general points in $H$.
Write $\Gamma = {C'_2} \cap {C''}$.

By Theorem~1.4 of~\cite{ibe} together with \eqref{inter}, the bundle $N_{C'_2}(-1)$ satisfies interpolation.
In particular, using \eqref{r2}, we have $H^1(N_{C'_2}(-1)(-\Delta)) = 0$
where $\Delta \subset {C'_2}$ is a set of $r + 2$ general points on ${C'_2}$.
We have the exact sequences
\begin{gather*}
0 \to N_{{C'_2} \cup {C''}}|_{C'_2}(-\Gamma - \Delta) \to N_{{C'_2} \cup {C''}}(-\Delta) \to N_{{C'_2} \cup {C''}}|_{C''} \to 0 \\
0 \to N_{C'_2}(-1)(-\Delta) \to N_{{C'_2} \cup {C''}}|_{C'_2}(-\Gamma - \Delta) \to * \to 0 \\
0 \to N_{{C''}/H} \to N_{{C'_2} \cup {C''}}|_{C''} \to N_H|_{C''}(\Gamma) \simeq \oo_{C''}(1)(\Gamma) \to 0,
\end{gather*}
where the $*$s denote punctual sheaves, which in particular have vanishing $H^1$.
Since Lemma~3.2 of \cite{rbn} gives $H^1(N_{{C''}/H}) = 0$,
we conclude $H^1(N_{{C'_2} \cup {C''}}(-\Delta)) = 0$
provided $H^1(\oo_{C''}(1)(\Gamma)) = 0$.
But since ${C''} \subset H$ is a general BN-curve of degree $d - d'$
and genus $g + 1 - g' - n$, we have either $H^0(\oo_{C''}(1)) = r$
or $H^1(\oo_{C''}(1)) = 0$. In the second case,
$H^1(\oo_{C''}(1)(\Gamma)) = 0$ is immediate. In the first case, this implies
via \eqref{interior} that
\[\dim H^1(\oo_{C''}(1)) = r - \chi(\oo_{C''}(1)) = n - 2 - (2n + d + g' - d' - g - r - 2) \leq n.\]
Since twisting up by a general point drops the dimension
of $H^1$ when that dimension is positive
(the Serre dual of the familiar statement that twisting down
by a general point drops the dimension of $H^0$ when that dimension is
positive), we conclude $H^1(\oo_{C''}(1)(\Gamma)) = 0$
and thus $H^1(N_{{C'_2} \cup {C''}}(-\Delta)) = 0$.

If $m = 0$, we take ${C'_1} = \emptyset$;
as $H^1(N_{{C'_2} \cup {C''}}(-\Delta)) = 0$,
we have
\[H^1(N_{{C'_1} \cup {C'_2} \cup {C''}}) = H^1(N_{{C'_2} \cup {C''}}) = 0.\]
In particular, $[{C'_1} \cup {C'_2} \cup {C''} = {C'_2} \cup {C''}]$ is a smooth point
of the Hilbert scheme (or equivalently the corresponding immersion is a smooth point of the space
of stable maps),
and is thus an interior curve as desired.

If $m \geq 1$, we apply Lemma~\ref{lm:forprop3} to let
${C'_1}$ be a BN-curve of degree
$rm$ and genus $(r + 1)(m - 1)$ whose intersection with ${C'_2}$ is exactly $\Delta$.
We then deform ${C'_1}$ to be general in some component of the space of BN-curves
passing through $\Delta$.
By Theorem~1.6 of~\cite{rbn}, both ${C'_1} \cup {C'_2}$ and ${C'_1} \cup {C'_2} \cup {C''}$ are BN-curves.
Moreover, using our assumption that $H^1(N_{{C'_2} \cup {C''}}(-\Delta)) = 0$,
Lemmas~3.2, 3.3, and 3.4 of \cite{rbn} imply ${C'_1} \cup {C'_2} \cup {C''}$
is an interior curve, as desired.
\end{proof}

In Appendix~\ref{app:inequalities}, we include code in {\sc sage}
to check or create algebraic expressions for all inequalities that appear in this section.

\section{Space Curves \label{sec:space}}

In this section, we prove Theorem~\ref{main} for space curves
($r = 3$); this will serve as one of the base cases for our larger
inductive argument.
Our argument here will also be by induction, this time on $n$.
Recall that, as mentioned in the introduction,
it suffices to prove maximal rank for $T$, and we may take $\epsilon = 0$ (c.f.\ \eqref{eps3}).

Our base case will be $n = 0$. The result here is a
well-known classical fact
if $C$ is a general rational curve of degree $d < r$
(in this case the restriction map is always surjective).
When $C$ is a BN-curve, the result is due to
Ballico and Ellia \cite{ball};
however, to make the paper as self-contained as possible,
we give an independent (and simpler) proof of this fact in Appendix~\ref{app:space}.

We therefore assume for our inductive argument that $n \geq 1$.
Write $d_i = \deg D_i$ and $g_i = \operatorname{genus} D_i$,
and suppose without loss of generality
that $d_1 \geq d_2 \geq \cdots \geq d_n$;
write $d$ and $g$ for the degree and genus of $C$.

By our inductive hypothesis, the subscheme
\[T_{n-1} := C \cup ((D_1 \cup D_2 \cup \cdots \cup D_{n - 1}) \cap \pp^r) \subset \pp^r\]
satisfies maximal rank for polynomials of degrees $k$ and $k - 1$.
Moreover, if $\Lambda \subset \pp^3$ is a general plane, and $k \neq 2$, then Theorem~1.3
of \cite{hyp} implies $T_{n-1} \cap \Lambda = C \cap \Lambda \subset \Lambda$
satisfies maximal rank for polynomials of degree $k$.
Since $T_{n-1}$ is positive-dimensional, an application of Theorem~1.5 
of \cite{hyp} completes the proof unless $k \geq 3$ and $(d_n, g_n) \in \{(8, 5), (9, 6), (10, 7)\}$, and
\[\dim H^0(\oo_{\pp^3}(k - 1)) > \dim H^0(\oo_{T_{n-1}}(k - 1)) \tand \dim H^0(\oo_\Lambda(k)) < 8 + \dim H^0(\oo_{C \cap \Lambda}(k)),\]
or equivalently, 
\begin{align}
\binom{k + 2}{3} &\geq (k - 1)d - g + 2 + \sum_{i = 1}^{n-1} d_i  \label{foo} \\
\binom{k + 2}{2} &\leq 7 + d. \label{bar}
\end{align}
Note also that in this case, $D_n \cap \pp^3$ is the general complete intersection
of $11 - d_n$ quadrics (c.f.\ Theorem~1.6 of \cite{quadrics});
in particular, we may specialize it to the union of the general complete intersection
of $3$ quadrics plus $d_n - 8$ additional general points.
We may thus reduce to the case $d_n = 8$.

Since $k \geq 3$, the inequality \eqref{bar} implies $d \geq \binom{k + 2}{2} - 7 \geq \binom{3 + 2}{2} - 7 = 3$;
in particular, $C$ cannot be a degenerate rational curve. We thus have $\rho(d, g, 3) \geq 0$,
or upon rearrangement:
\begin{equation} \label{gupper}
g \leq \frac{4d - 12}{3}.
\end{equation}
Combining \eqref{gupper} with our assumption that
$d_i \geq d_n \geq 8$ for all $i$,
condition \eqref{foo} implies
\[\binom{k + 2}{3} \geq \frac{3k - 7}{3} d + 8n - 2.\]
Rearranging and combining this with \eqref{bar}, we obtain
\begin{equation} \label{meh}
\frac{k^2 + 3k - 12}{2} \leq d \leq \frac{k^3 + 3 k^2 + 2 k + 12 - 48 n}{6k - 14}.
\end{equation}
In particular, if $n \geq 2$, then
\[\frac{k^2 + 3k - 12}{2} \leq \frac{k^3 + 3 k^2 + 2 k - 84}{6k - 14},\]
which does not hold for any $k \geq 3$; consequently, $n = 1$.
In this case,
\[\frac{k^2 + 3k - 12}{2} \leq \frac{k^3 + 3 k^2 + 2 k - 36}{6k - 14},\]
which does not hold for any $k \geq 5$; consequently, $k \in \{3, 4\}$.

For each $k$, equation \eqref{meh} gives upper and lower bounds on $d$;
for each such $d$, equations \eqref{gupper} and \eqref{foo} then give upper
and lower bounds respectively on $g$.

Using these bounds,
it thus remains only to consider the case where $n = 1$, and $D_1 = D_n$
is a canonical curve; and either $k = 3$ and
\[(d, g) \in \{(3, 0), (4, 0), (4, 1), (5, 2), (6, 4)\},\]
or $k = 4$ and
\[(d, g) = (8, 6).\]

If $k = 3$ and $(d, g) \in \{(4, 0), (5, 2), (6, 4)\}$,
then using our inductive hypothesis that $C$ satisfies maximal rank for cubics,
we see that $C$ lies on at most a $7$ dimensional family of cubics.
But since $D_1 \cap \pp^3$ is a general complete intersection of $3$ quadrics,
it contains $7$ general points.
Consequently, $C \cup (D_1 \cap \pp^3)$ does not lie on any cubics,
and so satisfies maximal rank for cubics.

If $k = 3$ and $(d, g) \in \{(3, 0), (4, 1)\}$, then partition $D_1 \cap \pp^3 = A \cup B$
into two sets of $4$ points, and let $Q$ be a general quadric containing $A$.
Since any subset of $7$ points of $D_1 \cap \pp^3$ are general, $Q$ is smooth and does not contain any point of $B$;
moreover, $A$ is a set of $4$ general points on $Q$, while $B$ is a set of $4$ general points in $\pp^3$
(although not independently general from $A$!).
We now specialize $C$ to a curve of type $(a, 2)$ on $Q$,
where
\[a = \begin{cases}
1 & \text{if $(d, g) = (3, 0)$;} \\
2 & \text{if $(d, g) = (4, 1)$.}
\end{cases}\]
The exact sequence of sheaves
\[0 \to \ii_{B \subset \pp^3} (1) \to \ii_{A \cup B \cup C \subset \pp^3} (3) \to \ii_{A \subset Q} (3 - a, 1) \to 0\]
gives rise to the long exact sequence in cohomology
\[\cdots \to H^1(\ii_{B \subset \pp^3} (1)) \to H^1(\ii_{A \cup B \cup C \subset \pp^3} (3)) \to H^1(\ii_{A \subset Q} (3 - a, 1)) \to \cdots.\]
Since $A$ and $B$ are general sets of $4$ points in $Q$ and $\pp^3$ respectively,
$\dim H^0(\oo_{\pp^3} (1)) = 4$ and $\dim H^0(\oo_Q(3 - a, 1)) = 8 - 2a \geq 4$,
while $H^1(\oo_{\pp^3} (1)) = H^1(\oo_Q(3-a, 1)) = 0$,
we conclude that $H^1(\ii_{B \subset \pp^3} (1)) = H^1(\ii_{A \subset Q} (3 - a, 1)) = 0$,
which implies $H^1(\ii_{A \cup B \cup C \subset \pp^3} (3)) = 0$ as desired.

If $k = 4$ and $(d, g) = (8, 6)$, then partition $D_1 \cap \pp^3 = A \cup B$
into two sets of $4$ points, and let $S$ be a general cubic containing $A$.
As in the previous cases, $S$ is smooth and does not contain any point of $B$;
moreover, $A$ is a set of $4$ general points on $S$, while $B$ is a set of $4$ general points in $\pp^3$.
Write $S$ as the blowup of $\pp^2$ at six points, $L$ for the pullback of the class of a line in $\pp^2$ to $S$,
and $E_1, E_2, \ldots, E_6$ for the six exceptional divisors.
By Lemma 9.3 of \cite{quadrics} plus results of \cite{keem}, we may specialize $C$ to a curve on $S$ of class
\[6L - E_1 - E_2 - 2E_3 - 2E_4 - 2E_5 - 2E_6.\]
The exact sequence of sheaves
\[0 \to \ii_{B \subset \pp^3} (1) \to \ii_{A \cup B \cup C \subset \pp^3} (3) \to \ii_{A \subset S} (3L - 2E_1 - 2E_2 - E_3 - E_4 - E_5 - E_6) \to 0\]
gives rise to the long exact sequence in cohomology
\[\cdots \to H^1(\ii_{B \subset \pp^3} (1)) \to H^1(\ii_{A \cup B \cup C \subset \pp^3} (4)) \to H^1(\ii_{A \subset S} (6L - 3E_1 - 3E_2 - 2E_3 - 2E_4 - 2E_5 - 2E_6)) \to \cdots.\]
Since the cone of effective curves on $S$ is spanned by the $27$ lines,
the Nakai-Moishezon criterion implies $9L - 4E_1 - 4E_2 - 3E_3 - 3E_4 - 3E_5 - 3E_6$ is ample.
Consequently, by Kodaira vanishing, $6L - 3E_1 - 3E_2 - 2E_3 - 2E_4 - 2E_5 - 2E_6$
has no higher cohomology.
In particular, by the Riemann--Roch theorem for surfaces, \mbox{$\dim H^0(\oo_S(6L - 3E_1 - 3E_2 - 2E_3 - 2E_4 - 2E_5 - 2E_6)) = 4$}.
Since $A$ and $B$ are general sets of $4$ points in $S$ and $\pp^3$ respectively,
we conclude that
\[H^1(\ii_{B \subset \pp^3} (1)) = H^1(\ii_{A \subset Q} (6L - 3E_1 - 3E_2 - 2E_3 - 2E_4 - 2E_5 - 2E_6)) = 0,\]
which implies $H^1(\ii_{A \cup B \cup C \subset \pp^3} (4)) = 0$ as desired,
thus completing the inductive step.

\section{Quadrics \label{sec:quadrics}}

In this section, we prove Theorem~\ref{main} for quadrics ($k = 2$);
this will serve as another base case for our larger inductive argument.
Our argument here will be by induction on $r$, and for fixed
$r$ by induction on $d$ --- using a construction
due to Ballico in \cite{ball2},
although our proof here will be logically independent from \cite{ball2}.
The base case of $r = 3$ was done in Section~\ref{sec:space},
so we suppose for our inductive argument that $r \geq 4$.
As noted in the introduction, it suffices to show $T$ satisfies maximal rank subject to \eqref{epsgeq4}.

If $g \geq d + 2$, then $C$
cannot be a degenerate rational curve, and so $C$ is a BN-curve
which satisfies
\begin{gather*}
(r + 1)(d - r) - r(g - r - 1) - r(r + 1) = (r + 1)d - rg - r(r + 1) \geq 0 \\
g - r - 1 \geq g - r - 1 - (r + 1)(g - d - 2) - [(r + 1)d - rg - r(r + 1)] = r^2 + 2r + 1 \geq 0.
\end{gather*}
We may therefore let $C'$ be a BN-curve of degree $d - r$
and genus $g - r - 1$ in $\pp^r$.
Applying Theorem~1.6 of~\cite{rbn} and Lemma~\ref{lm:forprop3},
we may degenerate $C$ to the union of $C'$ and an $(r + 2)$-secant
rational normal curve.
Note that the degree and genus of $C'$ satisfy
\begin{align*}
2(d - r) + 1 - (g - r - 1) &\geq 2(d - r) + 1 - (g - r - 1) - (r - 1)(g - d - 2) \\
&\qquad - [(r + 1)d - rg - r(r + 1)] \\
&= r^2 + 2r \geq \binom{r + 2}{2}.
\end{align*}
Thus by our inductive hypothesis,
$C'$ does not lie on any quadrics.
In particular, we conclude that $C$ and therefore $T$
also does not lie on any quadrics,
as desired.

We thus suppose $g \leq d + 1$ for the remainder of this section.
In these cases,
we pick a hyperplane $H \subset \pp^{r + 1}$, transverse to $\pp^r \subset \pp^{r + 1}$,
and write $\Lambda = \pp^r \cap H$ for the corresponding hyperplane in $\pp^r$.

If $C$ is a degenerate rational curve, we degenerate $C$ to a general rational curve in $\Lambda$,
and invoke Corollary~4.2 of \cite{hyp} to degenerate
each $D_i$ to a reducible curve $D_i' \cup D_i''$
with the $D_i' \subset \pp^{r + 1}$ and $D_i'' \subset H$ each sets of independently general
BN-curves or rational normal curves, which satisfy $\sum \deg D_i' = r + 1$.
With $(k, d', g', h') = (2, 0, 1, r + 1)$,
we note that \eqref{upstairs}
is an equality; using \eqref{epsgeq4}, this implies \eqref{downstairs} holds too.
In particular, $r$, $d$, $g$, and $h$ satisfy $I(r, 2, d, g, 0, 1, h, r + 1)$.
Applying Proposition~\ref{hir} thus yields the desired result.

Otherwise, if $C$ is a BN-curve, we first
invoke Corollary~4.2 of \cite{hyp} to degenerate
each $D_i$ to a general BN-curve $D_i^\circ \subset H$.

If $C$ is nonspecial, write $s = \min(r, g + 1)$.
Let $d'' = d - r$ and $g'' = g + 1 - s$.
Then we have either $g'' \geq 0$ and $d'' \geq g + r - 1$,
or $g'' = 0$ and $0 \leq s - 1 \leq d'' \leq r - 2$.
Therefore we can find a nonspecial BN-curve or degenerate rational curve
$Y \subset \Lambda$, of degree~$d''$
and genus~$g''$, which passes through $s$ general points in $\Lambda$.
We may thus construct a reducible curve $C' \cup C'' \subset \pp^r$,
where $C'' \subset \Lambda \subset \pp^r$ is as above,
$C' \subset \pp^r$ is a general rational normal curve
transverse to $\Lambda$, meeting $C''$ at $s$ general points in $\Lambda$.
This curve has degree $d'' + r = d$ and genus $g'' + s - 1 = g$,
and is a BN-curve by Theorem~1.6 of \cite{rbn};
this curve is thus a specialization of~$C$.
With $(k, d', g', h') = (2, r, 0, 0)$,
we note that \eqref{upstairs}
is an equality; using \eqref{epsgeq4}, this implies \eqref{downstairs} holds too.
In particular, $r$, $d$, $g$, and $h$ satisfy $I(r, 2, d, g, r, 0, h, 0)$.
Applying Proposition~\ref{hir} thus yields the desired result.

If $C$ is special (so $g + r - d - 1 \geq 0$),
then writing $d'' = d - r - 1$ and $g'' = g - r - 1$, we have
\begin{align*}
g'' &= [(r + 1)d - rg - r(r + 1)] + (r + 1)(g + r - d - 1) \geq 0 \\
\rho(d'', g'', r - 1) &= [(r + 1)d - rg - r(r + 1)] + (g + r - d - 1) \geq 0.
\end{align*}
We may therefore let $C'' \subset \Lambda$ be a general BN-curve of degree
$d''$ and genus $g''$.
Since $\operatorname{Aut} \Lambda$
acts transitively on collections of $r + 1$ points in linear general position,
we may therefore construct a reducible curve $C' \cup C'' \subset \pp^r$,
where $C'' \subset \Lambda \subset \pp^r$ is as above,
$C' \subset \pp^r$ is a general BN-curve
of degree $r + 1$ and genus $1$ transverse to $\Lambda$,
and $C' \cap \Lambda = C' \cap C''$ is a set of $r + 1$ general points in $\Lambda$.
This curve has degree $d'' + r + 1 = d$ and genus $g'' + r + 1 = g$,
and is a BN-curve by Theorem~1.9 of \cite{rbn} since our assumption that
$g \leq d + 1$ implies $d'' + (r + 1) \geq g'' + r$;
this curve is thus a specialization of $C$.
With $(k, d', g', h') = (2, r + 1, 1, 0)$,
we note that \eqref{upstairs}
is an equality;
using \eqref{epsgeq4}, this implies \eqref{downstairs} holds too.
In particular, $r$, $d$, $g$, and $h$ satisfy $I(r, 2, d, g, r + 1, 1, h, 0)$.
Applying Proposition~\ref{hir} thus yields the desired result.

\section{\boldmath Curves of Extreme Degree \label{extp}}

In this section, we deal with the ``easy'' cases of Theorem~\ref{main};
i.e.\ with those cases where the 
restriction map is ``far'' from being an isomorphism.
Since the cases of space curves and quadrics
have already been considered in the previous two sections,
we suppose $r \geq 4$ and $k \geq 3$.

\begin{defi}
We say that integers $r$, $k$, $d$, $g$, and $h$
\emph{satisfy $U(r, k, d, g, h)$} if
\begin{align*}
kd + 1 - g + h - \binom{r + k}{k} &\geq 0 \\
\frac{r}{r + k} \cdot \binom{r + k}{k} - d - \frac{r}{r + 1} \cdot h + r - 1 &\geq 0\\
g &\geq 0 \\
(r + 1) d - r g - r (r + 1) &\geq 0.
\end{align*}
We say that integers $r$, $k$, $d$, $g$, and $h$
\emph{satisfy $V(r, k, d, g, h)$} if they satisfy $U(r, k, d, g, h)$
and 
\[(r + 1)g + rh - r^2 - rk + r - k + 2 \geq 0;\]
while we say they \emph{satisfy $W(r, k, d, g, h)$} if they satisfy $U(r, k, d, g, h)$
and 
\[(r + 1)g + rh - r^2 - rk + r - k + 2 \leq -1.\]
\end{defi}

In this section, we show it suffices to verify Theorem~\ref{main}
when $r$, $k$, $d$, $g$, and $h$ satisfy either $V(r, k, d, g, h)$
or $W(r, k, d, g, h)$ --- or equivalently that it suffices to verify
Theorem~\ref{main}
when $r$, $k$, $d$, $g$, and $h$ satisfy $U(r, k, d, g, h)$.

Since $g \geq 0$ and we may suppose \eqref{epsgeq4} holds,
it remains to verify
Theorem~\ref{main} when 
\eqref{epsgeq4} holds and $C$ is a degenerate rational curve,
or \eqref{epsgeq4} holds and
\begin{equation} \label{rem}
\frac{r}{r + k} \cdot \binom{r + k}{k} - d - \frac{r}{r + 1} \cdot h + r - 1 < 0.
\end{equation}

We first verify Theorem~\ref{main} when \eqref{epsgeq4} and \eqref{rem} hold.
Fix a hyperplane $H \subset \pp^r$.
Applying Proposition~\ref{prop:hyp},
we may degenerate each $D_i$
to curves $D_i^\circ$, and the points $\{p_1, p_2, \ldots, p_\epsilon\}$
to points $\{p_1^\circ, p_2^\circ, \ldots, p_\epsilon^\circ\}$,
with
\[\#(\{p_1^\circ, p_2^\circ, \ldots, p_\epsilon^\circ\} \cap H) + \sum \# (D_i^\circ \cap H) = h - h',\]
so that
\begin{gather*}
\{p_1^\circ, p_2^\circ, \ldots, p_\epsilon^\circ\} \cap H \subset H \quad \text{and} \quad \\
\{p_1^\circ, p_2^\circ, \ldots, p_\epsilon^\circ\} \cap \pp^r \smallsetminus H \subset \pp^r
\end{gather*}
are sets of general points, and
\begin{gather*}
D_1^\circ \cap H, D_2^\circ \cap H, \ldots, D_n^\circ \cap H \subset H \quad \text{and} \\
D_1^\circ \cap \pp^r \smallsetminus H, D_2^\circ \cap \pp^r \smallsetminus H, \ldots, D_n^\circ \cap \pp^r \smallsetminus H \subset \pp^r
\end{gather*}
are sets of subsets of hyperplane sections of independently general
BN-curves, provided that $r$, $h$, and $h'$ satisfy $J(r, h, h')$ or $K(r, h, h')$.
Moreover, we can assume the second of these sets is a set of subsets
of hyperplane sections of independently general
nonspecial BN-curves provided that $r$, $h$, and $h'$ satisfy $N(r, h, h')$ or $K(r, h, h')$.

Applying Proposition~\ref{hir}, it thus remains to show there exists an integer $h'$
so that $r$, $k$, $d$, $g$, $h$, and $h'$ satisfy
\[I(r, k, d, g, d, g, h, h') \tand \begin{cases}
K(r, h, h') & \text{if $h \leq 2r + 1$;} \\
N(r, h, h') & \text{if $h \geq 2r + 2$ and $k = 3$;} \\
J(r, h, h') & \text{if $h \geq 2r + 2$ and $k \geq 4$.}
\end{cases}\]
Since these upper and lower bounds on $h'$
are all integers, we just have to check:
\begin{align*}
\frac{k}{r + k} \cdot \binom{r + k}{k} - (k - 1)d - 1 + g &\leq d + h - \frac{r}{r + k} \cdot \binom{r + k}{k}; \\
\begin{cases}
r + \left \lfloor \frac{h}{r + 1} \right \rfloor & \text{if $h \geq 2r + 2$,} \\
0 & \text{otherwise}
\end{cases} &\leq \begin{cases} 
h - r \cdot \left\lfloor \frac{h}{2r + 2}\right\rfloor & \text{if $h \geq 2r + 2$ and $k = 3$,} \\
h & \text{otherwise;}
\end{cases}
\\
\frac{k}{r + k} \cdot \binom{r + k}{k} - (k - 1)d - 1 + g &\leq \begin{cases} 
h - r \cdot \left\lfloor \frac{h}{2r + 2}\right\rfloor & \text{if $h \geq 2r + 2$ and $k = 3$,} \\
h & \text{otherwise;}
\end{cases} \\
\begin{cases}
r + \left \lfloor \frac{h}{r + 1} \right \rfloor & \text{if $h \geq 2r + 2$,} \\
0 & \text{otherwise}
\end{cases} &\leq d + h - \frac{r}{r + k} \cdot \binom{r + k}{k}.
\end{align*}
The first of these is just \eqref{epsgeq4} upon rearrangement.
The second is immediate by separately considering the cases $h \geq 2r + 2$
(which implies
$r + \frac{h}{r + 1} \leq h - r \cdot \frac{h}{2r + 2} \leq h$)
and $0 \leq h \leq 2r + 1$.
The final of these inequalities
follows from \eqref{rem}, which implies 
$r + \frac{h}{r + 1} < d + h - \frac{r}{r + k} \cdot \binom{r + k}{k} + 1$.

It thus remains to show the third of these inequalities.
By separately considering the cases when $C$ is a degenerate
rational curve and when $C$ is a BN-curve, we always have
\[(r + 1)d - rg - (r + 1) \geq 0.\]
Adding $kr - 2r - 1$ times
\eqref{epsgeq4}, we obtain upon rearrangement
\begin{align*}
\frac{k}{r + k} \cdot \binom{r + k}{k} - (k - 1)d - 1 + g &\leq \frac{kr - 2r - 1}{kr - r - 1} \cdot h - \frac{r(rk - 2r - k - 1) \cdot \binom{r + k}{k} + (2r + 1)(r + k)}{(kr - r - 1)(r + k)} \\
&\leq \frac{kr - 2r - 1}{kr - r - 1} \cdot h.
\end{align*}
It thus remains to show
\[\frac{kr - 2r - 1}{kr - r - 1} \cdot h \leq \begin{cases} 
h - r \cdot \left\lfloor \frac{h}{2r + 2}\right\rfloor & \text{if $h \geq 2r + 2$ and $k = 3$;} \\
h & \text{otherwise.}
\end{cases}\]
Since $h - r \cdot \left\lfloor \frac{h}{2r + 2}\right\rfloor \geq \frac{r + 2}{2r + 2} \cdot h$,
this reduces in turn to
\[\frac{kr - 2r - 1}{kr - r - 1} \leq \begin{cases} 
\frac{r + 2}{2r + 2} & \text{if $k = 3$;} \\
1 & \text{otherwise.}
\end{cases}\]
This is clear, this completing the verification of Theorem~\ref{main}
when \eqref{epsgeq4} and \eqref{rem} hold.

It thus remains to verify Theorem~\ref{main} when \eqref{epsgeq4} holds
and $C$ is a degenerate rational curve, but \eqref{rem} does not hold.
In this case, \eqref{epsgeq4} gives
\[kd + 1 + h \geq \binom{r + k}{k}.\]
Multiplying this by $r$ and adding $kr - r - 1$ times the inequality $d \leq r - 1$,
we obtain
\begin{align*}
\frac{r}{r + k} \cdot \binom{r + k}{k} - d - \frac{r}{r + 1} \cdot h + r - 1 &\leq \frac{1}{(r + 1)(r + k)} \cdot \left(r (r + k) (rk - k + 1) - r(k - 1) \cdot \binom{r + k}{k}\right) \\
\intertext{Since $r \geq 4$ and $k \geq 3$,}
&\leq \frac{1}{(r + 1)(r + k)} \cdot \left(r (r + k) (rk - k + 1) - r(k - 1) \cdot \binom{r + k}{3}\right) \\
&= -\frac{r}{6(r + 1)}
\cdot \Big(
(r - 4)^2(k - 3) + 2(r - 4)(k - 3)^2 \\
&\qquad + (k - 3)^3 + 2(r - 4)^2 + 9(r - 4)(k - 3) \\
&\qquad + 13(k - 3)^2 + 4(r - 4) + 34(k - 3) \Big) \\
&\leq 0.
\end{align*}
Since by assumption \eqref{rem} does not hold in this case,
we must have equality everywhere. In particular,
$(r, k) = (4, 3)$, and $d = r - 1 = 3$, and $h = \binom{r + k}{k} - kd - 1 = 25$.
It thus remains to consider the case $(r, k, d, g, h) = (4, 3, 3, 0, 25)$.

In this case, we let $H \simeq \pp^3 \subset \pp^4$ be the hyperplane
containing $C$.
Applying Proposition~\ref{prop:hyp},
we may degenerate each $D_i$
to curves $D_i^\circ$, and the points $\{p_1, p_2, \ldots, p_\epsilon\}$
to points $\{p_1^\circ, p_2^\circ, \ldots, p_\epsilon^\circ\}$,
with
\[\#(\{p_1^\circ, p_2^\circ, \ldots, p_\epsilon^\circ\} \cap H) + \sum \# (D_i^\circ \cap H) = 15,\]
so that
\begin{gather*}
\{p_1^\circ, p_2^\circ, \ldots, p_\epsilon^\circ\} \cap H \subset H \quad \text{and} \quad \\
\{p_1^\circ, p_2^\circ, \ldots, p_\epsilon^\circ\} \cap \pp^4 \smallsetminus H \subset \pp^4
\end{gather*}
are sets of general points, and
\begin{gather*}
D_1^\circ \cap H, D_2^\circ \cap H, \ldots, D_n^\circ \cap H \subset H \quad \text{and} \\
D_1^\circ \cap \pp^4 \smallsetminus H, D_2^\circ \cap \pp^4 \smallsetminus H, \ldots, D_n^\circ \cap \pp^4 \smallsetminus H \subset \pp^4
\end{gather*}
are sets of subsets of hyperplane sections of independently general
nonspecial BN-curves.
Applying Proposition~\ref{hir} then yields the desired result.

\section{The Inductive Argument \label{sec:inductive}}

In this section, we give our inductive argument to prove Theorem~\ref{main}.
We begin with the case
when $r$, $k$, $d$, $g$, and $h$ satisfy $W(r, k, d, g, h)$.
In this case, specializing $\{p_1, p_2, \ldots, p_\epsilon\}$
to general points in $H$, and combining Propositions~\ref{hir}, \ref{prop:crack}, and \ref{prop-rat},
we see that it suffices to check:

\begin{lm} \label{lm:w}
Let $r \geq 4$, and $k \geq 3$, and $d$, $g$, and $h$ be integers satisfying
$W(r, k, d, g, h)$.
Then for every integer $h'$ satisfying $A(r, h, h')$,
there exist integers $d'$, $g'$, and $n$ satisfying
\[X(r, d, g, d', g', n, k - 2) \tand I(r, k, d, g, d', g', h, h').\]
\end{lm}
\begin{proof}
This will be deferred to Appendix~\ref{code:w}.
\end{proof}

Next, we consider the case of cubic polynomials ($k = 3$),
when $r$, $d$, $g$, and $h$ satisfy $V(r, 3, d, g, h)$.
In this case, combining Propositions~\ref{hir}, \ref{prop:hyp}, \ref{prop-rat} (with $t = 0$), and~\ref{prop2},
we see that it suffices to check:

\begin{lm} \label{lm:v3}
Let $r \geq 4$, and $d$, $g$, and $h$ be integers satisfying
$V(r, 3, d, g, h)$.
Then there exist integers $d'$, $g'$, $h'$, and $n$ satisfying
\begin{itemize}
\item $X(r, d, g, d', g', n, 0)$, $I(r, 3, d, g, d', g', h, h')$, and $K(r, h, h')$;
\item $X(r, d, g, d', g', n, 0)$, $I(r, 3, d, g, d', g', h, h')$, and $M(r, h, h')$;
\item $X(r, d, g, d', g', n, 0)$, $I(r, 3, d, g, d', g', h, h')$, and $N(r, h, h')$;
\item $Y(r, d, g, d', g', n)$, $I(r, 3, d, g, d', g', h, h')$, and $K(r, h, h')$;
\item $Y(r, d, g, d', g', n)$, $I(r, 3, d, g, d', g', h, h')$, and $L(r, h, h')$; or
\item $Y(r, d, g, d', g', n)$, $I(r, 3, d, g, d', g', h, h')$, and $N(r, h, h')$.
\end{itemize}
\end{lm}
\begin{proof}
This will be deferred to Appendix~\ref{code:v3}.
\end{proof}

Finally, we consider the case of polynomials of higher degree ($k \geq 4$),
when $r$, $k$, $d$, $g$, and $h$ satisfy $V(r, k, d, g, h)$.
In this case, combining Propositions~\ref{hir}, \ref{prop:hyp}, \ref{prop-rat} (with $t = 0$), and~\ref{prop2},
we see that it suffices to check:

\begin{lm} \label{lm:v-higher}
Let $r \geq 4$ and $k \geq 4$, and $d$, $g$, and $h$ be integers satisfying
$V(r, k, d, g, h)$.
Then either:
\begin{itemize}
\item There exist integers $d'$, $g'$, $h'$, and $n$ satisfying
\[X(r, d, g, d', g', n, 0), \quad I(r, k, d, g, d', g', h, h'), \quad \text{and} \quad J(r, h, h');\]
\item There exist integers $d'$, $g'$, $h'$, and $n$ satisfying
\[Y(r, d, g, d', g', n), \quad I(r, k, d, g, d', g', h, h'), \quad \text{and} \quad J(r, h, h');\]
\item There exist integers $d'$, $g'$, $h'$, and $n$ satisfying
\[Y(r, d, g, d', g', n), \quad I(r, k, d, g, d', g', h, h'), \quad \text{and} \quad K(r, h, h');\]
\item There exist integers $d'$, $g'$, $h'$, $n$, and $m$ satisfying
\[Z(r, d, g, d', g', n, m), \quad I(r, k, d, g, d', g', h, h'), \quad \text{and} \quad J(r, h, h');\]
or
\item There exist integers $d'$, $g'$, $h'$, $n$, and $m$ satisfying
\[Z(r, d, g, d', g', n, m), \quad I(r, k, d, g, d', g', h, h'), \quad \text{and} \quad K(r, h, h').\]
\end{itemize}
\end{lm}
\begin{proof}
This will be deferred to Appendix~\ref{code:v-higher}.
\end{proof}

\appendix

\section{The Union of a Rational Curve with Disjoint Lines \label{app:lines}}

In this appendix, we prove the following special case of Ballico and Ellia's theorem
\cite{linesgeq4}
which we shall need for this paper.

\begin{thm}[Special case of \cite{linesgeq4}] \label{thm:lines}
Let $C \subset \pp^r$ (for $r \geq 3$) be a general rational curve of degree $d - n$,
and $L_1, L_2, \ldots, L_n \subset \pp^r$
be independently general lines.
Then $T \colonequals C \cup L_1 \cup \cdots \cup L_n$
satisfies maximal rank for polynomials of degree $k$ if $n \leq k - 1$.
\end{thm}

We argue by induction
on $r$ and $k$, using the result for $(r, k - 1)$ and $(r - 1, k)$
to prove the result for $(r, k)$; 
we note that the base case of $r = 3$ is a special case of
a result of Hartshorne and Hirschowitz \cite{lines3},
while the base case of $k = 1$ (which forces $n = 0$) is obvious.
We thus suppose $r \geq 4$ and $k \geq 2$ for our inductive argument.

\subsection{Degenerations \label{deg-lines}}

In this section, we outline the degenerations we shall use in
the proof of Theorem~\ref{thm:lines}.

\begin{defi}
We say that $r$, $k$, $d$, $n$, $d'$, and $n'$
\emph{satisfy $I_\ell(r, k, d, n, d', n')$}
(respectively \emph{satisfy $S_\ell(r, k, d, n, d', n')$})
if
\begin{align}
(k - 1) d' + 1 + n' &\geq \! (\text{respectively} \leq) \, \frac{k}{r + k} \cdot \binom{r + k}{k} \label{upstairs-ell} \\
kd + n - (k - 1)d' - n' &\geq \! (\text{respectively} \leq) \, \frac{r}{r + k} \cdot \binom{r + k}{k}. \label{downstairs-ell}
\end{align}
\end{defi}

\begin{prop} \label{hir-ell}
Let $r \geq 4$ and $H \subset \pp^r$ be a hyperplane.
Suppose that Theorem~\ref{thm:lines}
holds for $(r - 1, k)$ and $(r, k - 1)$.

Assume there exists a specialization of $C$ to $C^\circ = C' \cup C''$,
with $C''$ a general rational curve in $H$ and $C'$ transverse to $H$,
with $C' \cap C''$ general, and with $C'$ a general union of a rational curve
of degree $d' - n'$ with $n' \leq k - 2$ independently general lines.

If $r$, $k$, $d$, $n$, $d'$, and $n'$ satisfy
$I_\ell(r, k, d, n, d', n')$ or $S_\ell(r, k, d, n, d', n')$,
then Theorem~\ref{main} holds for $T$.
\end{prop}
\begin{proof}
Specialize each line $L_i$ to a general line $L_i^\circ \subset H$,
and write $T^\circ = C^\circ \cup L_1^\circ \cup L_2^\circ \cup \cdots \cup L_n^\circ$.
Then we have an exact sequence
\[0 \to \ii_{C' \subset \pp^r}(k - 1) \to \ii_{T^\circ \subset \pp^r}(k) \to \ii_{T^\circ \cap H \subset H}(k) \to 0.\]
Thus, to show $H^i(\ii_{T^\circ \subset \pp^r}(k)) = 0$,
it suffices to show $H^i(\ii_{C' \subset \pp^r}(k - 1)) = H^i(\ii_{T^\circ \cap H \subset H}(k)) = 0$,
or equivalently that
the restriction maps
\[H^0(\oo_H(k)) \to H^0(\oo_{T^\circ \cap H}(k)) \tand H^0(\oo_{\pp^r}(k - 1)) \to H^0(\oo_{C'}(k - 1))\]
are either both injective or both surjective.
The condition
$I_\ell(r, k, d, n, d', n')$, respectively $S_\ell(r, k, d, n, d', n')$,
implies
\[\dim H^0(\oo_{C'}(k)) \geq \dim H^0(\oo_{\pp^r}(k - 1)) \tand \dim H^0(\oo_{T^\circ \cap H}(k)) \geq \dim H^0(\oo_H(k)),\]
respectively
\[\dim H^0(\oo_{C'}(k)) \geq \dim H^0(\oo_{\pp^r}(k - 1)) \tand \dim H^0(\oo_{T^\circ \cap H}(k)) \geq \dim H^0(\oo_H(k)).\]

It thus remains to note that by our inductive hypothesis,
$C' \subset \pp^r$
satisfies maximal rank for polynomials of degree $k - 1$;
moreover,
$T^\circ \cap H \subset H$
is the union
of a subscheme which satisfies maximal rank for polynomials of degree $k$
with independently general points, and so satisfies maximal rank for polynomials of degree $k$ itself.
\end{proof}

\begin{defi}
We say that $r$, $d$, $n$, $d'$, $n'$, and $t$
\emph{satisfy $X_\ell(r, d, n, d', n', t)$} if
\begin{align}
n' &\geq 0 \label{aell}\\
t - n' &\geq 0 \label{bell}\\
d' - n' - 1 &\geq 0 \label{cell} \\
d - n - d' - 1 &\geq 0 \label{dell} \\
d - n - d' - n' &\geq 0. \label{eell}
\end{align}
\end{defi}

\begin{prop} \label{prop-rat-ell}
Let $r \geq 4$ and
$d$, $n$, $d'$, $n'$, and $t$ be integers
which satisfy $X_\ell(r, d, n, d', n', t)$.

Then there exists a rational curve $C' \cup C'' \subset \pp^r$
of degree $d - n$,
with $C''$ a general rational curve in a hyperplane $H$, and $C'$ a general
union of a rational curve of degree $d' - n'$ and $n' \leq t$ general lines,
all transverse to $H$.
\end{prop}
\begin{proof}
Write $d'' = d - n - d'$; note that adding \eqref{aell} and \eqref{cell} gives $d' \geq 1$,
and \eqref{dell} gives $d'' \geq 1$.
We may therefore let $C'' \subset H$ be a general rational curve
of degree $d''$.
By \eqref{eell}, which upon rearrangement becomes $n' + 1 \leq d'' + 1$,
we have that $C''$ passes through a set $\Gamma$ of $n' + 1$
general points in $H$.
Since $d' \geq n' + 1$ by \eqref{cell},
and the hyperplane section of a general union of rational curves of total
degree $d'$
is a general set of $d'$ points by Corollary~1.5 of~\cite{tan},
we may thus let $C'$ be a general union of a rational curve of degree $d' - n'$
and $n' \leq t$ general lines,
passing through $\Gamma$,
such that every component of $C'$ meets $\Gamma$.

By construction, $C' \cup C''$ is a rational curve of degree $d' + d'' = d$.
(c.f.\ Corollary~1.11 of~\cite{rbn} for its smoothability).
\end{proof}

In Appendix~\ref{app:inequalities}, we include code in {\sc sage}
to check or create algebraic expressions for all inequalities that appear in this section.

\subsection{Curves of Extreme Degree}

As in Section~\ref{extp}, we first deal with the
``easy'' cases of Theorem~\ref{thm:lines};
i.e.\ with those cases where the 
restriction map is ``far'' from being an isomorphism.

\begin{defi}
We say that integers $r$, $k$, $d$, and $n$ \emph{satisfy $U_\ell(r, k, d, n)$} if
\begin{align*}
n &\geq 0 \\
k - 1 - n &\geq 0 \\
(k - 1)(d - n) - \frac{k}{r + k} \cdot \binom{r + k}{k} &\geq 0 \\
\frac{r}{r + k} \cdot \binom{r + k}{k} - d - kn - 1 &\geq 0.
\end{align*}
\end{defi}

In this section, we show it suffices to verify Theorem~\ref{thm:lines} when
$r$, $k$, $d$, and $n$ satisfy $U_\ell(r, k, d, n)$.

To do this, we take $C^\circ = C$ in Proposition~\ref{hir-ell}, which reduces our problem to showing
that $r$, $k$, $d$, and $n$ satisfy $I_\ell(r, k, d, n, d - n, 0)$ or
$S_\ell(r, k, d, n, d - n, 0)$
provided that
\[d + kn \geq \frac{r}{r + k} \cdot \binom{r + k}{k} \tor (k - 1)(d - n) + 1 \leq \frac{k}{r + k} \cdot \binom{r + k}{k}.\]
Or equivalently, we need to show that
\[d + kn > \frac{r}{r + k} \cdot \binom{r + k}{k} \imp (k - 1)(d - n) + 1 \geq \frac{k}{r + k} \cdot \binom{r + k}{k}.\]
To see this, we simply note that $d + kn \geq \frac{r}{r + k} \cdot \binom{r + k}{k} + 1$ implies,
by adding $k + 1$ times our assumption that $n \leq k - 1$ and rearranging, that
\[(k - 1)(d - n) + 1 \geq \frac{k}{r + k} \cdot \binom{r + k}{k} + \frac{(rk - r - k) \cdot \binom{r + k}{k} - (r + k)(k^3 - k^2 - 2k + 1)}{r + k}.\]
It thus remains to see that for $k \geq 2$ and $r \geq 4$,
\[\binom{r + k}{k} \geq \frac{(r + k)(k^3 - k^2 - 2k + 1)}{rk - r - k}.\]
This is straight-forward for $k = 2$ and for $k = 3$; if $k \geq 4$, then
$\binom{r + k}{k} \geq \binom{r + k}{4}$, and so it suffices to see
\[\binom{r + k}{4} \geq \frac{(r + k)(k^3 - k^2 - 2k + 1)}{rk - r - k},\]
which holds because the difference is equal to
\begin{multline*}
\frac{r + k}{24(rk - r - k)} \cdot \Big((r - 4)^4(k - 4) + 3(r - 4)^3(k - 4)^2 + 3(r - 4)^2(k - 4)^3 + (r - 4)(k - 4)^4 \\
+ 3(r - 4)^4 + 30(r - 4)^3(k - 4) + 54(r - 4)^2(k - 4)^2 + 30(r - 4)(k - 4)^3 + 3(k - 4)^4 \\
+ 62(r - 4)^3 + 293(r - 4)^2(k - 4) + 293(r - 4)(k - 4)^2 + 38(k - 4)^3 + 465(r - 4)^2 \\
+ 1140(r - 4)(k - 4) + 201(k - 4)^2 + 1486(r - 4) + 574(k - 4) + 696 \Big) \geq 0.
\end{multline*}

\subsection{The Inductive Argument \label{inductive-lines}}

To prove Theorem~\ref{thm:lines}, we combine Propositions~\ref{hir-ell}
and \ref{prop-rat-ell} (with $t = k - 2$) to conclude that it suffices to check:

\begin{lm} \label{lm:lines}
Let $r \geq 4$ and $k \geq 2$, and $d$, $g$, and $n$ be integers satisfying
$U_\ell(r, k, d, g, h)$.
Then there exist integers $d'$ and $n'$ satisfying
\begin{itemize}
\item $X_\ell(r, d, n, d', n', k - 2)$ and $I_\ell(r, k, d, n, d', n')$; or
\item $X_\ell(r, d, n, d', n', k - 2)$ and $S_\ell(r, k, d, n, d', n')$.
\end{itemize}
\end{lm}
\begin{proof}
This will be deferred to Appendix~\ref{code:lines}.
\end{proof}

\section{Space Curves \boldmath ($n = 0$) \label{app:space}}

In this appendix, we give an alternative proof of Ballico and Ellia's theorem in
\cite{ball}:

\begin{thm} \label{thm:space}
Let $C \subset \pp^3$ be a general BN-curve of degree $d$ and genus $g$.
Then $C$ satisfies maximal rank for polynomials of any degree $k$.
\end{thm}

As in \cite{ball}, our argument will be by induction on $k$.
The base cases of $k = 0$ and $k = 1$ are clear: Since $C$ is nonempty, respectively nondegenerate, the restriction map is injective.
We therefore suppose $k \geq 2$ for the remainder of this section.
As in \cite{ball}, we will use an inductive argument specializing $C$ to a reducible curve $C' \cup C''$
with $C''$ lying on a smooth quadric $Q \simeq \pp^1 \times \pp^1$.

\subsection{Degenerations \label{deg-space}}

In this section, we outline the degenerations we shall use in the proof of Theorem~\ref{thm:space}.

\begin{defi}
We say that $k$, $d$, $g$, $d'$, and $g'$
\emph{satisfy $I_3(k, d, g, d', g')$} (respectively
\emph{satisfy $S_3(k, d, g, d', g')$}) if
\begin{align}
(k - 2) d' + 1 - g' &\geq \! (\text{respectively} \leq) \ \binom{k + 1}{3} \label{upstairs3}\\
kd - g - (k - 2)d' + g' &\geq \! (\text{respectively} \leq) \ (k + 1)^2. \label{downstairs3}
\end{align}
\end{defi}

\begin{prop} \label{hir3}
Let $Q \subset \pp^3$ be a smooth quadric, and
suppose that Theorem~\ref{thm:space} holds for polynomials of degree $k - 2$.

Assume there exists
a specialization $C^\circ = C' \cup C''$ of $C$,
with $C''$ contained in $Q$ and $C'$ transverse to $Q$,
with $C' \cap Q$ (and thus $C' \cap C''$) a general set of points on $Q$,
with $C'$ either a general BN-curve
or the union of a general rational curve with independently general disjoint lines,
and with $C'$ of
degree $d'$ and genus $g'$.

If $k$, $d$, $g$, $d'$, and $g'$ satisfy
either
\[I_3(k, d, g, d', g') \tor S_3(k, d, g, d', g'),\]
then Theorem~\ref{thm:space} holds for $C$.
\end{prop}
\begin{proof}
We have an exact sequence
\[0 \to \ii_{C' \subset \pp^3}(k - 2) \to \ii_{C^\circ \subset \pp^3}(k) \to \ii_{C^\circ \cap Q \subset Q \simeq \pp^1 \times \pp^1}(k, k) \to 0.\]
Thus, to show $H^i(\ii_{C^\circ \subset \pp^3}(k)) = 0$,
it suffices to show $H^i(\ii_{C' \subset \pp^3}(k - 2)) = H^i(\ii_{C^\circ \cap Q \subset Q}(k)) = 0$.

Write $(a, b)$ for the bidegree of $C'' \subset Q \simeq \pp^1 \times \pp^1$.
Write $\Gamma = C^\circ \cap Q \smallsetminus C''$, so that
$\ii_{C^\circ \cap Q \subset Q}(k) \simeq \ii_{\Gamma \subset \pp^1 \times \pp^1}(k - a, k - b)$.
Writing $n = \# \Gamma$, we note that \eqref{downstairs3} becomes upon rearrangement
\[n \geq \! (\text{respectively} \leq) \, (k + 1 - a)(k + 1 - b).\]

If $k$, $d$, $g$, $d'$, and $g'$ satisfy
$I_3(k, d, g, d', g')$,
then since 
general points impose independent conditions
on sections of any line bundle, and
\[\dim H^0(\oo_{\pp^1 \times \pp^1}(k - a, k - b)) \leq \max(0, (k + 1 - a)(k + 1 - b)),\]
we conclude that $H^0(\ii_{\Gamma \subset \pp^1 \times \pp^1}(k - a, k - b)) = 0$.
Moreover, the vanishing of $H^0(\ii_{C' \subset \pp^r}(k - 2))$ is equivalent to the injectivity
of the restriction map
\[H^0(\oo_{\pp^3}(k - 2)) \to H^0(\oo_{C'}(k - 2)),\]
which is of maximal rank by our inductive hypothesis
(if $C'$ is a BN-curve)
or \cite{lines3}
(if $C'$ is the union of a general rational curve with independently general disjoint lines).
It thus remains to note that \eqref{upstairs3} implies
\[\chi(\oo_{C'}(k - 2)) = (k - 2) d' + 1 - g' \geq \binom{k + 1}{3} = \dim H^0(\oo_{\pp^3}(k - 2)),\]
and so 
$\dim H^0(\oo_{C'}(k - 2)) \geq \chi(\oo_{C'}(k - 2)) \geq \dim H^0(\oo_{\pp^3}(k - 2))$
as desired.

We now suppose $k$, $d$, $g$, $d'$, and $g'$ satisfy $S_3(k, d, g, d', g')$,
but not $I_3(k, d, g, d', g')$. Adding together \eqref{upstairs3} and \eqref{downstairs3},
one of which must be strict,
we obtain
\begin{equation} \label{from-s3}
kd + 1 - g \leq \binom{k + 1}{3} + (k + 1)^2 - 1.
\end{equation}
Since $C^\circ$ must be connected, we have
\begin{equation} \label{conn}
2d' - n \geq 1 - g'.
\end{equation}
Finally, since $C$ is a BN-curve, we have
\begin{equation} \label{rho}
4d - 3g - 12 \geq 0.
\end{equation}
Adding together $4$ times \eqref{from-s3} plus
$3k - 4$ times \eqref{conn} plus $k$ times \eqref{rho},
and using $d = a + b + d'$ and $g = ab - a - b + 2d' + g' - n$,
we obtain upon rearrangement
\[(a - 1)(b - 1) \leq \frac{2k^3 + 12k^2 - 14k - 12}{9k - 12} \leq k^2.\]
We may therefore suppose without loss of generality that
\begin{equation} \label{a-bound}
a \leq k + 1.
\end{equation}

By separately considering the cases $k = 2$ (for which \eqref{upstairs3} implies $g' \geq 0$ and thus $C'$ is connected),
$k = 3$ (for which \eqref{upstairs3} implies $d' \leq g' + 3$ and thus $C'$ is linearly normal),
and $k \geq 4$ (for which $\oo_{C'}(k - 2)$ is nonspecial),
we conclude
\[\dim H^0(\oo_{C'}(k - 2)) \leq \dim H^0(\oo_{\pp^3}(k - 2)).\]
Since $C' \subset \pp^3$ satisfies maximal rank for polynomials of degree $k - 2$
by our inductive hypothesis
(if $C'$ is a BN-curve)
or \cite{lines3}
(if $C'$ is the union of a general rational curve with independently general disjoint lines),
this implies the restriction map
$H^0(\oo_{\pp^3}(k - 2)) \to \dim H^0(\oo_{C'}(k - 2))$ is surjective, and therefore that $H^1(\ii_{C' \subset \pp^3}(k - 2)) = 0$.

It thus remains to show $H^1(\ii_{\Gamma \subset \pp^1 \times \pp^1}(k - a, k - b)) = 0$,
for which we will use \eqref{a-bound}.

If $b \geq k + 1$,
then since $n \leq (k + 1 - a)(k + 1 - b) \leq 0$ by assumption,
we must have $n = 0$ and either $a = k + 1$ or $b = k + 1$.
It follows that
\[H^1(\ii_{\Gamma \subset \pp^1 \times \pp^1}(k - a, k - b)) = H^1(\oo_{\pp^1 \times \pp^1}(k - a, k - b)) = 0,\]
as desired. On the other hand, if $b \leq k + 1$,
then $H^1(\oo_{\pp^1 \times \pp^1}(k - a, k - b)) = 0$, so
the vanishing of $H^1(\ii_{\Gamma \subset \pp^1 \times \pp^1}(k - a, k - b))$ is equivalent to the surjectivity
of the restriction map
\[H^0(\oo_{\pp^1 \times \pp^1}(k - a, k - b)) \to H^0(\oo_\Gamma(k - a, k - b)),\]
which holds since general points impose independent conditions on sections of any line bundle
and $n \leq (k + 1 - a)(k + 1 - b) = \dim H^0(\oo_{\pp^1 \times \pp^1}(k - a, k - b))$ by assumption.
\end{proof}

\begin{defi}
We define the set
\[E_3 \colonequals \{(4, 1), (5, 2), (6, 2), (6, 4), (7, 5), (8, 6)\}.\]
\end{defi}

\begin{defi}
We say that $d$, $g$, $d'$, and $g'$
\emph{satisfy $X_3(d, g, d', g')$} if
\begin{align}
d' + g' - 3 &\geq 0 \label{a3} \\
4 d' - 3 g' - 12 &\geq 0 \label{b3} \\
d - d' - 1 &\geq 0 \label{c3} \\
g - g' &\geq 0 \label{d3} \\
2 d - 2 d' - g + g' - 2 &\geq 0 \label{e3}\\
2 d' - g + g' - 1 &\geq 0. \label{f3}
\end{align}
\end{defi}

\begin{prop} \label{prop-rat-3}
Let $Q \subset \pp^3$ be a smooth quadric, and
$d$, $g$, $d'$, and $g'$ be integers which satisfy
\begin{align*}
g &\geq 0 \\
4 d - 3g - 12 &\geq 0 \\
(d', g') &\notin E_3,
\end{align*}
and $X_3(d, g, d', g')$.

Then there exists a BN-curve $C' \cup C'' \subset \pp^3$
of degree $d$ and genus $g$,
with $C''$ contained in $Q$ and $C'$ transverse to $Q$,
with $C' \cap Q$ a general set of points on $Q$,
with $C'$ either a general BN-curve
or the union of a general rational curve with independently general disjoint lines,
and with $C'$ of
degree $d'$ and genus $g'$.
\end{prop}
\begin{proof}
Using \eqref{b3} if $g' \geq 0$, or \eqref{a3} if $g' \leq 0$,
there exists a curve $C' \subset \pp^3$ of degree $d'$ and genus $g'$,
which is either a general BN-curve
or the union of a general nondegenerate rational curve with independently general disjoint lines.
Since $(d', g') \notin E_3$, and a line meets $Q$ in $2$ general points, Theorem~1.4 of~\cite{quadrics}
implies $C' \cap Q$ is a general collection of $2d'$ points on $Q$.

From \eqref{c3}, we may let $C'' \subset Q$ be a general curve of bidegree $(1, d - d' - 1)$,
which is smooth and irreducible,
is rational by adjunction, and is of degree $d'' = d - d'$. Write $n = g - g' + 1$.

Upon rearrangement, \eqref{d3} becomes $n \geq 1$.
From \eqref{e3}, which upon rearrangement becomes $n \leq 2(d - d') - 1 = \dim H^0(\oo_{\pp^1 \times \pp^1}(1, d - d' - 1)) - 1$,
the curve $C''$ passes through $n$ general points in $Q$.
From \eqref{f3}, which upon rearrangement becomes $n \leq 2d'$,
the curve $C'$ also passes through $n$ general points in $Q$,
which we can assume consists of at least one point in each of the $\max(1, 1 - g')$ components of $C'$,
as $g \geq 0$ and $n \geq 1$ imply $\max(1, 1 - g') \leq n$.
We may thus arrange for $C'$ to meet $C''$ at a set $\Gamma$ of exactly $n$ points, which are general in $Q$,
so that $C' \cup C''$ is connected of degree $d' + d'' = d$ and genus $g' + n - 1 = g$.

Upon rearrangement, \eqref{e3} becomes $n \leq 2d'' - 1$.
By Theorem~1.1 of~\cite{vogt}, the curve $C''$ can therefore
be deformed to pass through $n$ general points.
Moreover, since $H^1(N_{C'}(-\Gamma)) = 0$ by Lemma~3.1 of~\cite{rbn},
any deformation of $\Gamma$ can be lifted to a deformation of $C'$.
We may thus deform $C'$ and $C''$ while preserving the incidence conditions,
so they meet at a general set of $n$ points in $\pp^3$.
We conclude that $C' \cup C''$ is a BN-curve by Theorem~1.5 of~\cite{rbn2},
in combination with \eqref{e3} which rules out the case $n = 2d' = 2d''$,
and our assumption that $\rho(d, g, 3) = 4d - 3g - 12 \geq 0$.
\end{proof}

\begin{defi}
We say that $d$, $g$, $d'$, and $g'$
\emph{satisfy $Y_3(d, g, d', g')$} if
\begin{align}
g' &\geq 0 \label{g3} \\
4 d' - 3 g' - 12 &\geq 0 \label{h3} \\
d - d' - 3 &\geq 0 \label{i3} \\
g - d + d' - g' + 3 &\geq 0 \label{j3} \\
3 d - g - 3 d' + g' - 5 &\geq 0 \label{k3}\\
d - g + d' + g' - 4 &\geq 0. \label{l3}
\end{align}
\end{defi}

\begin{prop} \label{prop2-3}
Let $Q \subset \pp^3$ be a smooth quadric, and
$d$, $g$, $d'$, and $g'$ be integers which satisfy
\begin{align*}
g &\geq 0 \\
4 d - 3g - 12 &\geq 0 \\
(d', g') &\notin E_3,
\end{align*}
and $Y_3(d, g, d', g')$.

Then there exists a BN-curve $C' \cup C'' \subset \pp^3$
of degree $d$ and genus $g$,
with $C''$ contained in $Q$ and $C'$ transverse to $Q$,
with $C' \cap Q$ a general set of points on $Q$,
and with $C'$ a general BN-curve of
degree $d'$ and genus $g'$.
\end{prop}
\begin{proof}
Using \eqref{g3} and \eqref{h3}, we may let
$C' \subset \pp^3$ be a general BN-curve of degree $d'$ and genus $g'$.
Since $(d', g') \notin E_3$, Theorem~1.4 of~\cite{quadrics}
implies $C' \cap Q$ is a general collection of $2d'$ points.

From \eqref{i3}, we may let $C'' \subset Q$ be a general curve of bidegree $(2, d - d' - 2)$,
which is smooth, irreducible, and nondegenerate, is of genus $g'' = d - d' - 3$ by adjunction, and is of degree $d'' = d - d'$.
Note that $d'' = g'' + 3$, so $C''$ is a BN-curve by results of \cite{keem};
also, $(d'', g'') \neq (6, 4)$.
Write $n = g - d + d' - g' + 4$.

Upon rearrangement, \eqref{j3} becomes $n \geq 1$.
Adding \eqref{i3} and \eqref{k3}, we obtain upon rearrangement
$n \leq 3(d - d' - 1) - 1 = \dim H^0(\oo_{\pp^1 \times \pp^1}(2, d - d' - 2)) - 1$,
and so $C''$ passes through $n$ general points in $Q$.
From \eqref{l3}, which upon rearrangement becomes $n \leq 2d'$,
the curve $C'$ also passes through $n$ general points in $Q$.
We may thus arrange for $C'$ to meet $C''$ at a set $\Gamma$ of exactly $n$ points, which are general in $Q$,
so that $C' \cup C''$ is connected (since $n \geq 1$),
of degree $d' + d'' = d$ and genus $g' + g'' + n - 1 = g$.

Upon rearrangement, \eqref{k3} becomes $n \leq 2d'' - 1$.
Since $C''$ is a BN-curve and $(d'', g'') \neq (6, 4)$,
Theorem~1.1 of~\cite{vogt} implies $C''$
can therefore
be deformed to pass through $n$ general points.
Moreover, since $H^1(N_{C'}(-\Gamma)) = 0$ by Lemma~3.1 of~\cite{rbn},
any deformation of $\Gamma$ can be lifted to a deformation of $C'$.
We may thus deform $C'$ and $C''$ while preserving the incidence conditions,
so they meet at a general set of $n$ points in $\pp^3$.
We conclude that $C' \cup C''$ is a BN-curve by Theorem~1.5 of~\cite{rbn2},
in combination with \eqref{k3} which rules out the case $n = 2d' = 2d''$,
and our assumption that $\rho(d, g, 3) = 4d - 3g - 12 \geq 0$.
\end{proof}

\subsection{Curves of Extreme Degree}

As in Section~\ref{extp}, we first deal with the
``easy'' cases of Theorem~\ref{thm:space};
i.e.\ with those cases where the 
restriction map is ``far'' from being an isomorphism.

\begin{defi}
We say that integers $k$, $d$, and $g$ \emph{satisfy $U_3(k, d, g)$} if
\begin{align*}
g &\geq 0 \\
4 d - 3 g - 12 &\geq 0 \\
(k - 1) d - g - \binom{k + 2}{3} &\geq 0 \\
\binom{k + 2}{2} - d - 1 &\geq 0.
\end{align*}
\end{defi}

In this section, we show it suffices to verify Theorem~\ref{thm:space} when
$k$, $d$, and $g$ satisfy $U_3(k, d, g)$.
If $(d, g) = (6, 4)$, it is a classical fact that $C$ is the complete intersection of a quadric
and cubic surface, from which the conclusion of Theorem~\ref{thm:space} may easily be verified.
We thus suppose for this section that $(d, g) \neq (6, 4)$, and that $k$, $d$, and $g$
do not satisfy $U_3(k, d, n)$.

In this case, Theorem~1.5 of~\cite{quadrics} implies, for a plane $H \subset \pp^3$,
that $C \cap H$ is a general set of $d$ points in $H$.
We have an exact sequence
\[0 \to \ii_{C \subset \pp^3}(k - 1) \to \ii_{C \subset \pp^3}(k) \to \ii_{C \cap H \subset H}(k) \to 0.\]
Thus, to show $H^i(\ii_{C \subset \pp^3}(k)) = 0$,
it suffices to show $H^i(\ii_{C \subset \pp^3}(k - 1)) = H^i(\ii_{C \cap H \subset H}(k)) = 0$,
or equivalently that the restriction maps
\[H^0(\oo_H(k)) \to H^0(\oo_{C \cap H}(k)) \tand H^0(\oo_{\pp^3}(k - 1)) \to H^0(\oo_C(k - 1))\]
are either both injective or both surjective.
Since these maps are of maximal rank by the generality of $C \cap H$ and our inductive hypothesis
respectively, it suffices to show that either
\[\dim H^0(\oo_H(k)) \geq \dim H^0(\oo_{C \cap H}(k)) \tand \dim H^0(\oo_{\pp^3}(k - 1)) \geq \dim H^0(\oo_C(k - 1)),\]
or
\[\dim H^0(\oo_H(k)) \leq \dim H^0(\oo_{C \cap H}(k)) \tand \dim H^0(\oo_{\pp^3}(k - 1)) \leq \dim H^0(\oo_C(k - 1)).\]

If $k = 2$ and $d \leq g + 3$,
then $\dim H^0(\oo_C(k - 1)) = 4 = \dim H^0(\oo_{\pp^3}(k - 1))$, so we are done;
we may thus suppose $d \geq g + 3$ if $k = 2$, so
$\dim H^0(\oo_C(k - 1)) = (k - 1)d + 1 - g$.
It therefore remains to show that in this case, either
\[d \leq \binom{k + 2}{2} \tand (k - 1)d + 1 - g \leq \binom{k + 2}{3},\]
or
\[d \geq \binom{k + 2}{2} \tand (k - 1)d + 1 - g \geq \binom{k + 2}{3}.\]
But since $k$, $d$, and $g$ do not satisfy $U_3(k, d, n)$, we in particular have
\[d \geq \binom{k + 2}{2} \tor (k - 1)d + 1 - g \leq \binom{k + 2}{3}.\]
It thus remains to show, under the assumption that $d \geq g + 3$ if $k = 2$, that
\[d > \binom{k + 2}{2} \imp (k - 1)d + 1 - g \geq \binom{k + 2}{3}.\]

If $k = 2$, the right-hand inequality becomes our assumption $d \geq g + 3$,
so it remains to consider the case $k \geq 3$.
In this case, multiplying the inequality $d \geq \binom{k + 2}{2} + 1$ by $3k - 7$,
adding $4d - 3g - 12 \geq 0$, and rearranging, we obtain
\[(k - 1)d + 1 - g \geq \frac{3k^3 + 2k^2 - 9k + 2}{6} \geq \binom{k + 2}{3},\]
as desired.

\subsection{The Inductive Argument \label{inductive-space}}

To prove Theorem~\ref{thm:space}, we combine Propositions~\ref{hir3},
\ref{prop-rat-3}, and \ref{prop2-3} to conclude that it suffices to check:

\begin{lm} \label{lm:space}
Let $k \geq 2$, and $d$ and $g$ be integers satisfying
$U_3(k, d, g)$.
Then there exist integers $d'$ and $g'$ satisfying
\begin{itemize}
\item $X_3(d, g, d', g')$, $I_3(k, d, g, d', g')$, and $(d', g') \notin E_3$;
\item $Y_3(d, g, d', g')$, $I_3(k, d, g, d', g')$, and $(d', g') \notin E_3$;
\item $X_3(d, g, d', g')$, $S_3(k, d, g, d', g')$, and $(d', g') \notin E_3$; or
\item $Y_3(d, g, d', g')$, $S_3(k, d, g, d', g')$, and $(d', g') \notin E_3$.
\end{itemize}
\end{lm}
\begin{proof}
This will be deferred to Appendix~\ref{code:space}.
\end{proof}

\section{Inequalities from Section~\ref{sec:degenerations} \label{app:inequalities}}

In this appendix, we give code in {\sc sage}
to check or create algebraic expressions for all inequalities that appear in
Sections~\ref{sec:degenerations}, \ref{deg-lines}, and \ref{deg-space}.

\begin{lstlisting}
###### File name: inequalities.py ######


from sage.all import binomial


### From Section 2: ###


def I(r, k, d, g, dp, gp, h, hp, B):
	# We pass B = binomial(r + k, k) as an extra argument.
	return [
		(k - 1) * dp + 1 - gp + hp - k * B / (r + k),
		k * d - g - (k - 1) * dp + gp + h - hp - r * B / (r + k)
	]


def S(r, k, d, g, dp, gp, h, hp, B):
	# We pass B = binomial(r + k, k) as an extra argument.
	return [
		k * B / (r + k) - ((k - 1) * dp + 1 - gp + hp),
		r * B / (r + k) - (k * d - g - (k - 1) * dp + gp + h - hp)
	]


def A(r, h, hp):
	return [
		hp,
		h - (r + 1) * hp
	]


def J(r, h, hp):
	return [
		hp - r - h / (r + 1),
		h - hp
	]


def K(r, h, hp):
	return [
		2 * r + 1 - h,
		hp,
		h - hp
	]


def L(r, h, hp):
	return [
		3 * r + 2 - h,
		h - 2,
		hp - 2,
		2 - hp
	]


def M(r, h, hp):
	return [
		3 * r + 2 - h,
		hp - r - 2,
		h - hp - r/2
	]


def N(r, h, hp):
	return [
		h,
		hp - r - h / (r + 1),
		h - hp - r * h / (2 * r + 2)
	]


def Np(r, h, hp):
	# When 3r + 3 <= h <= 4r + 3:
	#    then floor(h / (r + 1)) = 3 and floor(h / (2r + 2)) = 1,
	#    so N(r, h, hp) is satisfied if:
	return [
		h - 3 * r - 3,
		4 * r + 3 - h,
		hp - r - 3,
		h - hp - r
	]


def X(r, d, g, dp, gp, n, t):
	return [
		-gp,
		dp + gp - 1,
		t + gp,
		g - gp - n + 1,
		r * (d - dp) - (r - 1) * (g - gp) + (r - 1) * n - r**2 + 1,
		n + gp - 1,
		dp - n,
		r * (d - dp) - (r - 4) * (g - gp) - 2 * n - 2 * r + 2,
		r + 2 - n - gp,
		2 * n + d + gp - dp - g - r - 1
	]


def Y(r, d, g, dp, gp, n):
	return [
		gp,
		(r + 1) * dp - r * gp - r**2 - r,
		(2 * r - 3) * dp - (r - 2)**2 * gp - 2 * r**2 + 3 * r - 9,
		g - gp - n + 1,
		r * (d - dp) - (r - 1) * (g - gp) + (r - 1) * n - r**2 + 1,
		n - 1,
		dp - n,
		r * (d - dp) - (r - 4) * (g - gp) - 2 * n - 2 * r + 2,
		2 * n + d + gp - dp - g - r - 2
	]


def Z(r, d, g, dp, gp, m, n):
	return [
		gp - (r + 1) * m,
		(r + 1) * dp - r * gp - r**2 - r,
		(2 * r - 3) * (dp - r * m) - (r - 2)**2 * (gp - (r + 1) * m) - 2 * r**2 + 3 * r - 9,
		g - gp - n + 1,
		r * (d - dp) - (r - 1) * (g - gp) + (r - 1) * n - r**2 + 1,
		n - 1,
		(dp - r * m) - n,
		r * (d - dp) - (r - 4) * (g - gp) - 2 * n - 2 * r + 2,
		2 * n + d + gp - dp - g - r - 2,
		2 * (dp - r * m) - (r - 3) * (gp - (r + 1) * m - 1) - (r - 1) * (r + 2),
		m
	]


def U(r, k, d, g, h, B):
	# We pass B = binomial(r + k, k) as an extra argument.
	return [
	k * d + 1 - g + h - B,
	r * B / (r + k) - d - r * h / (r + 1) + r - 1,
	g,
	(r + 1) * d - r * g - r * (r + 1),
	h
]


def V(r, k, d, g, h, B):
	# We pass B = binomial(r + k, k) as an extra argument.
	return U(r, k, d, g, h, B) + [(r + 1) * g + r * h - r**2 - r * k + r - k + 2]


def W(r, k, d, g, h, B):
	# We pass B = binomial(r + k, k) as an extra argument.
	return U(r, k, d, g, h, B) + [-1 - ((r + 1) * g + r * h - r**2 - r * k + r - k + 2)]


### From Appendix A.1: ###


def I_ell(r, k, d, n, dp, np, B):
	# We pass B = binomial(r + k, k) as an extra argument.
	return [
		(k - 1) * dp + 1 + np - k * B / (r + k),
		k * d + n - (k - 1) * dp - np - r * B / (r + k)
	]


def S_ell(r, k, d, n, dp, np, B):
	# We pass B = binomial(r + k, k) as an extra argument.
	return [
		k * B / (r + k) - ((k - 1) * dp + 1 + np),
		r * B / (r + k) - (k * d + n - (k - 1) * dp - np)
	]


def X_ell(r, d, n, dp, np, t):
	return [
		np,
		t - np,
		dp - np - 1,
		d - n - dp - 1,
		d - n - dp - np
	]

def U_ell(r, k, d, n, B):
   # We pass B = binomial(r + k, k) as an extra argument.
   return [
	n,
	k - 1 - n,
	(k - 1) * (d - n) - k * B / (r + k),
	r * B / (r + k) - d - k * n - 1
]


### From Appendix B.1: ###


def I_3(k, d, g, dp, gp):
   return [
      (k - 2) * dp + 1 - gp - binomial(k + 1, 3),
      k * d - g - (k - 2) * dp + gp - (k + 1)**2
   ]


def S_3(k, d, g, dp, gp):
   return [
      binomial(k + 1, 3) - ((k - 2) * dp + 1 - gp),
      (k + 1)**2 - (k * d - g - (k - 2) * dp + gp)
   ]


E_3 = [(4, 1), (5, 2), (6, 2), (6, 4), (7, 5), (8, 6)]


def X_3(d, g, dp, gp):
	return [
		dp + gp - 3,
		4 * dp - 3 * gp - 12,
		d - dp - 1,
		g - gp,
		2 * d - 2 * dp - g + gp - 2,
		2 * dp - g + gp - 1
	]


def Y_3(d, g, dp, gp):
	return [
		gp,
		4 * dp - 3 * gp - 12,
		d - dp - 3,
		g - d + dp - gp + 3,
		3 * d - g - 3 * dp + gp - 5,
		d - g + dp + gp - 4
	]


def U_3(k, d, g):
	return [
		g,
		4 * d - 3 * g - 12,
		(k - 1) * d - g - binomial(k + 2, 3),
		binomial(k + 2, 2) - d - 1
	]
\end{lstlisting}

\section{Computations for Sections~\ref{sec:inductive}, \ref{inductive-lines}, and \ref{inductive-space} \label{app:comp}}

To carry out the computations in Section~\ref{sec:inductive},
we use code in {\sc sage} to study systems of inequalities which are linear
in most --- that is, all but zero, respectively one, respectively two, of the variables --- with coefficients that are
constants; respectively polynomials in the remaining variable;
respectively polynomials in the remaining two variables,
which we term $r$ and $k$, and the binomial coefficient $B \colonequals \binom{r + k}{k}$,
which are linear in $B$.
We term these remaining zero, one, or two variables the \emph{base variables}.

We represent such inequalities as linear polynomials over a \emph{base ring}
defined in {\sc sage}, of the form \\
\verb|QQ|, \\
respectively \\
\verb|QQ['x'].fraction_field()|, \\
respectively \\
\verb|QQ['r, k']['B'].fraction_field()|.

We suppose the coefficients of these polynomials (except for the constant coefficient)
are integer-valued (for any integer values of the base variables).

In Appendix~\ref{app:covers},
we define a function \texttt{covers}: This function takes
as input three such systems of inequalities, plus a list of lower bounds of length
zero, respectively one, respectively two.
It then attempts to prove that, subject to the given lower bounds on the base variables,
every assignment of integers satisfying the third system of inequalities may be
completed to an assignment of integers which would
satisfy either the first or second system
if all the constant coefficients appearing in the first or second
systems were rounded up (i.e.\ replaced by their ceiling).
When these systems are declared, we must order the variables for the first
two systems to begin with the variables occurring in the third.

\paragraph{Example:}
Suppose there is one base variable $r$, and let
\begin{align*}
X &= [g - d, r - g] \\
Y &= [g - 3, d - g, h - d, r^2 - h] \\
Z &= [d, r^2 - d].
\end{align*}
Suppose that $Z$ is thought of as a system of inequalities in $r$ and $d$,
represented in {\sc sage} as a list of elements of $\qq[r][d]$.
To represent
$X$ and $Y$, we must order the variables correctly, starting with the
variables for $Z$:
Namely, suppose that $X$, respectively $Y$, is thought of as a system of
inequalities in $r$, $d$, $g$, 
respectively in $r$, $d$, $g$, $h$; then this would be
represented in {\sc sage} as a list
of elements of $\qq[r][d, g]$, respectively $\qq[r][d, g, h]$:

\begin{lstlisting}
from cover import covers

R = QQ['r'].fraction_field()
r = R.gen()

d, g = R['d, g'].gens()
X = [g - d, r - g]

d, g, h = R['d, g, h'].gens()
Y = [g - 4, d - g, h - d, r**2 - h]

d = R['d'].gen()
Z = [d, r**2 - d]

print covers(X, Y, Z, [3])
print covers(X, Y, Z, [2])
\end{lstlisting}

Then \texttt{covers(X, Y, Z, [3])} will hopefully return \texttt{True} ---
since if $r \geq 3$ is an integer,
then for every integer $d$ with $0 \leq d \leq r^2$,
either there exists an integer $g$ with $d \leq g \leq r$,
or there exists integers $g$ and $h$ with $3 \leq g \leq d$ and $d \leq h \leq r^2$.
However, \texttt{covers(X, Y, Z, [2])} must return \texttt{False} ---
since if $r = 2$, then there is some integer $d$ (namely $d = 3$) for which
no such integer $g$, respectively no such integers $g$ and $h$, exist.
(Indeed, the above outputs \texttt{True} and \texttt{False} respectively,
in less than half a second.)

\bigskip

To complete the computations from Section~\ref{sec:inductive},
we run the \texttt{covers} function with the following inputs
(on an installation of {\sc sage 8.0} running on a 2.4 GHz CPU).

\subsection{For Lemma~\ref{lm:w} \label{code:w}}
Running the following code outputs \texttt{True} in about 20 minutes: 
\begin{lstlisting}
from cover import covers
from inequalities import X, I, W, A

R = QQ['r, k']['B'].fraction_field()
B = R.gen()
r, k = [R(i) for i in R.base_ring().gens()]

d, g, h, hp, dp, gp, n = R['d, g, h, hp, dp, gp, n'].gens()

xi = X(r, d, g, dp, gp, n, k - 2) + I(r, k, d, g, dp, gp, h, hp, B)

d, g, h, hp = R['d, g, h, hp'].gens()

wa = W(r, k, d, g, h, B) + A(r, h, hp)

print covers(xi, xi, wa, [4, 3])
\end{lstlisting}

\subsection{For Lemma~\ref{lm:v3} \label{code:v3}}
Running the following code outputs \texttt{True} in about 20 minutes: 
\begin{lstlisting}
from cover import covers
from inequalities import X, Y, I, V, K, L, M, N, Np

R = QQ['r'].fraction_field()
r = R.gen()

B = binomial(r + 3, 3)

d, g, h, dp, gp, hp, n = R['d, g, h, dp, gp, hp, n'].gens()

xi = X(r, d, g, dp, gp, n, 0) + I(r, 3, d, g, dp, gp, h, hp, B)
xik = xi + K(r, h, hp)
xim = xi + M(r, h, hp)
xin = xi + N(r, h, hp)
xinp = xi + Np(r, h, hp)

yi = Y(r, d, g, dp, gp, n) + I(r, 3, d, g, dp, gp, h, hp, B)
yik = yi + K(r, h, hp)
yil = yi + L(r, h, hp)
yin = yi + N(r, h, hp)
yinp = yi + Np(r, h, hp)

d, g, h = R['d, g, h'].gens()

target = V(r, 3, d, g, h, B)

target1 = target + [2 * r + 1 - h]
target2 = target + [h - 2 * r - 2, 3 * r + 2 - h]
target3 = target + [h - 3 * r - 3, 4 * r + 3 - h]
target4 = target + [h - 4 * r - 4]

c1 = covers(xik, yik, target1, [4])
c2 = covers(xim, yil, target2, [4])
c3 = covers(xinp, yinp, target3, [4])
c4 = covers(xin, yin, target4, [4])
print (c1 and c2 and c3 and c4)
\end{lstlisting}

\subsection{For Lemma~\ref{lm:v-higher} \label{code:v-higher}}
Running the following code outputs \texttt{True} in about 2 days: 
\begin{lstlisting}
from cover import covers
from inequalities import X, Y, Z, I, V, J, K

R = QQ['r, k']['B'].fraction_field()
B = R.gen()
r, k = [R(i) for i in R.base_ring().gens()]

d, g, h, dp, gp, hp, n = R['d, g, h, dp, gp, hp, n'].gens()

xij = X(r, d, g, dp, gp, n, 0) + I(r, k, d, g, dp, gp, h, hp, B) + J(r, h, hp)
yij = Y(r, d, g, dp, gp, n) + I(r, k, d, g, dp, gp, h, hp, B) + J(r, h, hp)
yik = Y(r, d, g, dp, gp, n) + I(r, k, d, g, dp, gp, h, hp, B) + K(r, h, hp)

d, g, h, dp, gp, hp, n, m = R['d, g, h, dp, gp, hp, n, m'].gens()

zij = Z(r, d, g, dp, gp, n, m) + I(r, k, d, g, dp, gp, h, hp, B) + J(r, h, hp)
zik = Z(r, d, g, dp, gp, n, m) + I(r, k, d, g, dp, gp, h, hp, B) + K(r, h, hp)

d, g, h = R['d, g, h'].gens()

target = V(r, k, d, g, h, B)

c = covers(yik, zik, target + [r - h], [4, 4])

target.append(h - r - 1)

c = c and covers(yij, zij, target, [5, 4])

S = R.base_ring().fraction_field()
r, k = S.gens()
f = Hom(R, S)(binomial(k + 4, 4))
T = QQ['k'].fraction_field()
k = T.gen()
g = Hom(S, T)([4, k])

target = [i.change_ring(f).change_ring(g) for i in target]
xij = [i.change_ring(f).change_ring(g) for i in xij]
yij = [i.change_ring(f).change_ring(g) for i in yij]

c = c and covers(xij, yij, target, [4])

print c
\end{lstlisting}

\subsection{For Lemma~\ref{lm:lines} \label{code:lines}}
Running the following code outputs \texttt{True} in about 30 seconds: 
\begin{lstlisting}
from cover import covers
from inequalities import X_ell, I_ell, S_ell, U_ell

R = QQ['r, k']['B'].fraction_field()
B = R.gen()
r, k = [R(i) for i in R.base_ring().gens()]

d, n, dp, np = R['d, n, dp, np'].gens()

xi = X_ell(r, d, n, dp, np, k - 2) + I_ell(r, k, d, n, dp, np, B)
xs = X_ell(r, d, n, dp, np, k - 2) + S_ell(r, k, d, n, dp, np, B)

d, n = R['d, n'].gens()

target = U_ell(r, k, d, n, B)

print covers(xi, xs, target, [4, 2])
\end{lstlisting}

\subsection{For Lemma~\ref{lm:space} \label{code:space}}
Running the following code outputs \texttt{True} in about 15 seconds: 
\begin{lstlisting}
from cover import covers
from inequalities import I_3, S_3, E_3, X_3, Y_3, U_3

R = QQ['k'].fraction_field()
k = R.gen()

d, g, dp, gp = R['d, g, dp, gp'].gens()

def f(dpi, gpi):
	return 2 * dpi - gpi

notE = [f(dp, gp) - max([f(dpi, gpi) for dpi, gpi in E_3]) - 1]

xi = X_3(d, g, dp, gp) + I_3(k, d, g, dp, gp)
yi = Y_3(d, g, dp, gp) + I_3(k, d, g, dp, gp)
xs = X_3(d, g, dp, gp) + S_3(k, d, g, dp, gp)
ys = Y_3(d, g, dp, gp) + S_3(k, d, g, dp, gp)

d, g = R['d, g'].gens()

target = U_3(k, d, g)

pivot = k * d + 1 - g - binomial(k + 3, 3)

c = covers(xi + notE, yi + notE, target + [pivot], [5]) and covers(xs + notE, ys + notE, target + [-1 - pivot], [5])

d, g, dp, gp = QQ['d, g, dp, gp'].gens()

for k in xrange(2, 5):
	fk = Hom(R, QQ)(k)

	xi_k = [i.change_ring(fk) for i in xi]
	yi_k = [i.change_ring(fk) for i in yi]
	xs_k = [i.change_ring(fk) for i in xs]
	ys_k = [i.change_ring(fk) for i in ys]

	target_k = [i.change_ring(fk) for i in target]
	target_k = Polyhedron(ieqs=[[i.constant_coefficient(), i.coefficient({d:1}), i.coefficient({g:1})] for i in target_k])

	for dt, gt in target_k.integral_points():
		xi_kdg = [i.subs({d:dt, g:gt}) for i in xi_k]
		yi_kdg = [i.subs({d:dt, g:gt}) for i in yi_k]
		xs_kdg = [i.subs({d:dt, g:gt}) for i in xs_k]
		ys_kdg = [i.subs({d:dt, g:gt}) for i in ys_k]

		xi_kdg = Polyhedron(ieqs=[[i.constant_coefficient(), i.coefficient({dp:1}), i.coefficient({gp:1})] for i in xi_kdg])
		yi_kdg = Polyhedron(ieqs=[[i.constant_coefficient(), i.coefficient({dp:1}), i.coefficient({gp:1})] for i in yi_kdg])
		xs_kdg = Polyhedron(ieqs=[[i.constant_coefficient(), i.coefficient({dp:1}), i.coefficient({gp:1})] for i in xs_kdg])
		ys_kdg = Polyhedron(ieqs=[[i.constant_coefficient(), i.coefficient({dp:1}), i.coefficient({gp:1})] for i in ys_kdg])

		ok = False
	
		for dpt, gpt in xi_kdg.integral_points() + yi_kdg.integral_points() + xs_kdg.integral_points() + ys_kdg.integral_points():
			if (dpt, gpt) not in E_3:
				ok = True
		
		if not ok:
			c = False

print c
\end{lstlisting}

\section{Code for the \texttt{covers} function \label{app:covers}}

To implement the \texttt{covers} function, our first strategy is to
compute, for a given system of inequalities of our form,
a new system of inequalities so that every assignment of integers satisfying
the new system may be completed to one satisfying the old system.
For this, we use the following result:

\begin{lm} \label{lm:eliminate}
Let $a_i$, $b_i$, $c_j$, and $d_j$ be sets of integers,
with the $b_i$ and $d_j$ positive.
There exists a real number, respectively integer, $n$ satisfying the inequalities
\[n \leq \frac{a_i}{b_i} \tand n \geq \frac{c_j}{d_j}\]
if, for each $(i, j)$,
\[a_i d_j - b_i c_j \geq 0 \quad \text{respectively} \quad a_i d_j - b_i c_j \geq (b_i - 1)(d_j - 1).\]
\end{lm}
\begin{proof}
Note that the collection of intervals $[\frac{c_j}{d_j}, \frac{a_i}{b_i}]$
are closed under intersection. Their intersection is thus nonempty, respectively
contains an integer, if and only if each such interval does.

It thus remains to show that if $a$, $b$, $c$, and $d$ are integers,
with $b$ and $d$ positive,
then there exists a real number, respectively integer, $n$ satisfying the inequalities
\begin{equation} \label{interval}
\frac{c}{d} \leq n \leq \frac{a}{b}
\end{equation}
if
\[a d - b c \geq 0 \quad \text{respectively} \quad a d - b c \geq (b - 1)(d - 1).\]

The first part is clear: There exists a real number $n$ satisfying \eqref{interval}
if and only if $\frac{c}{d} \leq \frac{a}{b}$, which upon rearrangement becomes
$a d - b c \geq 0$.
For the second part, we note that an integer $n$ satisfies \eqref{interval}
if and only if $n$ satisfies
\[\frac{c - 1}{d} < n < \frac{a + 1}{b}.\]
Since any interval in $\rr$ of length strictly greater than $1$
contains an integer,
such an integer $n$ exists provided that
\[\frac{a + 1}{b} - \frac{c - 1}{d} > 1,\]
which in turn holds provided that
\[\frac{a + 1}{b} - \frac{c - 1}{d} \geq 1 + \frac{1}{bd}.\]
Upon rearrangement, this gives the desired condition
$a d - b c \geq (b - 1)(d - 1)$.
\end{proof}

Iteratively applying Lemma~\ref{lm:eliminate},
we reduce our problem to checking that every assignment
of integers satisfying a given system of inequalities of our form
satisfies one of two other such systems of inequalities.

\begin{rem} \label{rem:order}
If the real numbers version of
Lemma~\ref{lm:eliminate} is applied iteratively to eliminate several
variables,
the region described by the resulting system of inequalities is independent
of the order of the variables to which Lemma~\ref{lm:eliminate} is applied.

However, if the integral version of Lemma~\ref{lm:eliminate} is applied
iteratively, this is no longer the case.
Since our implementation of
the \texttt{covers} function eliminates variables in the
reverse direction to the direction they were declared, the order
of declaration may affect the output (even subject to the constraints
on this order explained in Section~\ref{app:comp}).
\end{rem}

For given values of the base variables,
these systems of inequalities become convex polyhedra.
This problem can thus be solved using the following result:

\begin{lm} \label{lm:poly}
Let $P \subset \rr^n$ be a compact convex polyhedron, and $C_1, C_2 \subset \rr^n$ be
convex sets. Then $P \subset C_1 \cup C_2$ if and only if:
\begin{enumerate}
\item \label{vertex} Every vertex of $P$ is contained in either $C_1$ or $C_2$; and
\item \label{edge} Every edge of $P$ joining a vertex not contained in $C_1$ to a vertex not contained in $C_2$
meets $C_1 \cap C_2$.
\end{enumerate}
\end{lm}
\begin{rem}
To check Condition~\ref{edge},
it suffices to check that the every such edge, when extended to a line, meets $C_1 \cap C_2$.
Indeed, if one endpoint of the edge is not contained in $C_1$, and the other is not contained in $C_2$,
then $C_1 \cap C_2$ cannot meet this line outside the edge.
\end{rem}

\begin{proof}
The ``only if'' is clear. To see the ``if'',
we argue by induction on $n$. The result for $n = 0$ and $n = 1$ is clear.

When $n = 2$, we argue by induction on the number of vertices of $P$.
Since the result when $P$ is degenerate (i.e.\ has two or fewer vertices) follows from our inductive hypothesis,
we begin with the case when $P$ is a triangle.
If all three vertices are contained in the same convex set,
the result is clear. It thus suffices to consider the case when two vertices $x$ and $y$ are contained in $C_1$,
while the third vertex $z$ is contained in $C_2$.
By assumption $xz$ and $yz$ meet $C_1 \cap C_2$ at points $x'$ and $y'$ respectively.
The triangle $xyz$ then decomposes into the union of the quadrilateral $xyy'x'$ and the
triangle $x'y'z$,
which are contained in $C_1$ and $C_2$ respectively.

To finish the inductive argument for the case $n = 2$,
we first consider the case when $P$ has two non-adjacent vertices
which are contained in the same convex set. In this case,
the line joining these two vertices partitions $P$ into two polygons
with fewer vertices, each of which satisfies our inductive hypothesis.
It thus remains to consider the case when no two non-adjacent vertices
are contained in the same convex set. Since $P$ is not a triangle, this forces
all vertices contained in $C_1$ to lie on a single edge, and similarly for $C_2$.
Thus $P$ must be a quadrilateral $xyzw$, with two adjacent vertices $x$ and $y$ contained in $C_1$,
and the other two adjacent vertices $z$ and $w$ contained in $C_2$.
By assumption, the edges $yz$ and $xw$ meet $C_1 \cap C_2$ at points
$p$ and $q$ respectively.
The quadrilateral $xyzw$ then decomposes into the union of the quadrilateral $xypq$ and the quadrilateral
$pqwz$, which are contained in $C_1$ and $C_2$ respectively.

For the inductive argument on $n$, we first note that by our inductive hypothesis,
the boundary of $P$ is contained in $C_1 \cup C_2$.
It thus remains to show $x \in C_1 \cup C_2$,
where $x$ is an arbitrary point in the interior of $P$.
For this, we let $H$ be any hyperplane containing $x$.
Since $n \geq 3$, every vertex and every edge of $H \cap P$
lies on the boundary of $P$. Applying our inductive hypothesis again,
we conclude $x \in H \cap P \subset C_1 \cup C_2$ as desired.
\end{proof}

For any fixed value of the base variables, Lemmas~\ref{lm:eliminate}
and~\ref{lm:poly} can be used to show 
every assignment of integers satisfying the third system of inequalities may be
completed to an assignment of integers which would
satisfy either the first or second system
if all the constant coefficients appearing in the first or second
systems were rounded up (i.e.\ replaced by their ceiling), as desired.

However, in order to apply these results in a uniform way,
we need a result which determines the sign of certain polynomials.
For polynomials of one variable, this is straight-forward (just find the real roots).
For polynomials of two variables, we use the following result:

\begin{lm}
A polynomial $P(r, k)$
in two variables $r$ and $k$ is positive for all $r \geq r_0$
and $k \geq k_0$, provided that:
\begin{enumerate}

\item \label{epos} Every monomial on the outside of the Newton polygon has positive
coefficient.

\item \label{sposk} The leading coefficient with respect to $r$ is positive for $k \geq k_0$.

\item \label{sposr} The leading coefficient with respect to $k$ is positive for $r \geq r_0$.

\item \label{boundary} The polynomial $P(r_0, k)$ is positive for
$k \geq k_0$.

\item \label{branch} The value of $k_0$ exceeds all branch points of the projection
onto the $k$-axis of $P(r, k) = 0$.
\end{enumerate}
\end{lm}
\begin{proof}
Condition~\ref{epos} guarantees there is some $r_0'$ and $k_0'$
so that $P(r, k) > 0$ for $r \geq r_0'$ and $k \geq k_0'$.
Condition~\ref{sposk} guarantees that for every $k_0' \geq k_0$,
there is some $r_0''$ so that $P(r, k) > 0$ for $k_0 \leq k \leq k_0'$ and $r \geq r_0''$.
Condition~\ref{sposr} guarantees that for every $r_0' \geq r_0$,
there is some $k_0''$ so that $P(r, k) > 0$ for $r_0 \leq r \leq r_0'$ and $k \geq k_0''$.
Putting these together, we conclude that
\[R \colonequals \{(r, k) \in \rr^2 : r \geq r_0, \, k \geq k_0, \, \text{and} \, P(r, k) \leq 0\}\]
is compact.

Assume for sake of contradiction that
$R$ is nonempty. Then we may let $(r', k')$ be the point in $R$ with maximal $k$-coordinate;
necessarily, $P(r', k') = 0$.
Applying Lagrange multipliers, we conclude
that either $r' = r_0$ or $\left.\frac{\partial P}{\partial r}\right|_{(r, k) = (r', k')} = 0$;
these contradict Conditions~\ref{boundary} and~\ref{branch} respectively.
\end{proof}

Combining these techniques with brute-force search
where appropriate, we obtain the following implementation of the \texttt{covers} function:

\begin{lstlisting}
###### File name: cover.py ######


from sage.all import *


def subsets(L, n):
	# Generator to yield all subsets of L of cardinality n.

	if n == 0:
		yield []
		return
	
	for i in xrange(len(L)):
		for S in subsets(L[:i], n - 1):
			yield [L[i]] + S

	return


def bound_largest_root(P):
	# Bound largest root of a univariate polynomial P(x).

	if P == 0:
		raise ValueError
	elif P.is_constant():
		return -infinity
	else:
		# Strip multiple roots so we don't crash PARI.
		P = (P / P.derivative()).numerator()

	largest_root = max([i[0] for i in P.roots(RR)] + [-infinity])
	
	if largest_root == -infinity:
		return -infinity
	else:
		# Bound largest_root strictly from above (by an integer).
		return Integer(floor(largest_root)) + 1


def asymptotic_sign(P, bounds, factor=True):
	# Find sign of P when all base variables are large, plus
	# lower bounds on the base variables which ensure this sign is acheived.
	# If no such lower bounds exist, raise NotImplementedError.
	#
	# We try to find such bounds which are as good as we can, subject to
	# the constraint that they imply the inputted bounds (second argument).
	#
	# If factor is True, we (may) attempt to factor P recursively.

	R = P.parent()
	
	if R == QQ:
		return (sign(P), [])

	if R.is_field():
		s1, bounds = asymptotic_sign(P.numerator(), bounds)
		s2, bounds = asymptotic_sign(P.denominator(), bounds)
		return (s1 * s2, bounds)

	if P == 0:
		return (0, bounds)

	n = R.krull_dimension()

	if n == 1:
		sgn = sign(P.leading_coefficient())
		t = max(bound_largest_root(P), bounds[0])
		while (sign(P(t - 1)) == sgn) and (t - 1 >= bounds[0]):
			t -= 1

		return (sgn, [t])
				
	if n == 2:
		r, k = R.gens()
		if P.degree(k) == 0:
			s, p = asymptotic_sign(PolynomialRing(QQ, r)(P), [bounds[0]])
			return (s, p + [bounds[1]])
		elif P.degree(r) == 0:
			s, p = asymptotic_sign(PolynomialRing(QQ, k)(P), [bounds[1]])
			return (s, [bounds[0]] + p)
		
		if factor:
			Pf = P.factor()
			sgn = sign(QQ(Pf.unit()))
			for i, j in Pf:
				s, bounds = asymptotic_sign(i, bounds, factor=False)
				sgn *= s
			return (sgn, bounds)
		
		# Check that every monomial on the outside of the Newton polygon
		# has a coefficient of the same sign (and determine what that sign is).

		exponent_polyhedron = Polyhedron(vertices=P.exponents(), rays=[[-1, 0], [0, -1]])
		e1 = vector([1, 0])
		e2 = vector([0, 1])

		extremal_vertices = []
		for j in P.exponents():
			v = vector(j)
			if (not exponent_polyhedron.interior_contains(v)) and (not exponent_polyhedron.contains(v + e1)) and (not exponent_polyhedron.contains(v + e2)):
				extremal_vertices.append(j)

		signs = [sign(P.dict()[j]) for j in extremal_vertices]
	
		if 1 not in signs:
			sgn = -1
		elif -1 not in signs:
			sgn = 1
		else:
			raise NotImplementedError
		
		# By Lemma C.5, sign of P is sgn for r >= rmax and k >= kmax,
		# where rmax and kmax are defined by:

		rmax = max(bounds[0], bound_largest_root(P.coefficient({k : P.degree(k)}).univariate_polynomial()))
		kmax = max(bounds[1], bound_largest_root(P.coefficient({r : P.degree(r)}).univariate_polynomial()), bound_largest_root(P.discriminant(r).univariate_polynomial()), bound_largest_root(P.subs({r : rmax}).univariate_polynomial()))
		
		U = QQ['x']
		x = U.gen()
		
		# Brute force search for better bounds:

		t = rmax
		for t in xrange(rmax - 1, bounds[0] - 1, -1):
			s, p = asymptotic_sign(Hom(R, U)([QQ(t), x])(P), [kmax])

			if s != sgn:
				t += 1
				break
	
			for i in xrange(kmax, p[0]):
				s = sign(Hom(R, QQ)([QQ(t), QQ(i)])(P))
				if s != sgn:
					bounds[1] = max(bounds[1], i + 1)

		bounds[0] = t
	
		t = kmax
		for t in xrange(kmax - 1, bounds[1] - 1, -1):
			s, p = asymptotic_sign(Hom(R, U)([x, QQ(t)])(P), [bounds[0]])

			if s != sgn:
				t += 1
				break
			
			ok = True
			for i in xrange(bounds[0], p[0]):
				s = sign(Hom(R, QQ)([QQ(i), QQ(t)])(P))
				if s != sgn:
					t += 1
					ok = False
					break
			if not ok:
				break

		bounds[1] = t
		
		return (sgn, bounds)
	
	if n == 3:
		B = R.gen()
		if P.degree(B) == 0:
			return asymptotic_sign(P.constant_coefficient(), bounds)

		r, k = R.base_ring().gens()

		sgn, lc = asymptotic_sign(P.leading_coefficient(), bounds)

		# P is linear in B, so P increasing in B for r >= lc[0] and k >= lc[1].

		rmax, kmax = infinity, infinity
		m = 0
		while 1:
			try:
				s, (c_rmax, c_kmax) = asymptotic_sign(P.subs({B:binomial(r + k, m)}), [max(m, lc[0]), max(m, lc[1])])
				if s == sgn:
					# Since B >= binomial(r + k, m) for r >= m and k >= m,
					# sign of P is sgn for r >= c_rmax and k >= c_kmax.
					if c_rmax + c_kmax < rmax + kmax:
						rmax, kmax = c_rmax, c_kmax
						raise NotImplementedError
					else:
						break
				else:
					raise NotImplementedError
			except NotImplementedError:
				m += 1

		# Know sign of P is sgn for r >= rmax and k >= kmax.
		# Now, we do brute force search for better bounds:

		U = QQ['x']
		x = U.gen()
		S = R.base_ring()
		
		t = rmax
		for t in xrange(rmax - 1, bounds[0] - 1, -1):
			B0 = binomial(k + t, t)
			s, p = asymptotic_sign(Hom(S, U)([QQ(t), x])(Hom(R, S)(B0)(P)), [kmax])

			if s != sgn:
				t += 1
				break
	
			for i in xrange(kmax, p[0]):
				s = sign(Hom(S, QQ)([QQ(t), QQ(i)])(Hom(R, S)(B0)(P)))
				if s != sgn:
					bounds[1] = max(bounds[1], i + 1)

		bounds[0] = t
	
		t = kmax
		for t in xrange(kmax - 1, bounds[1] - 1, -1):
			B0 = binomial(r + t, t)
			s, p = asymptotic_sign(Hom(S, U)([x, QQ(t)])(Hom(R, S)(B0)(P)), [bounds[0]])

			if s != sgn:
				t += 1
				break
			
			ok = True
			for i in xrange(bounds[0], p[0]):
				s = sign(Hom(S, QQ)([QQ(i), QQ(t)])(Hom(R, S)(B0)(P)))
				if s != sgn:
					t += 1
					ok = False
					break
			if not ok:
				break

		bounds[1] = t
		
		return (sgn, bounds)
	
	raise NotImplementedError


def coef(P, v):
	# Return coefficient of v in P, with uniform interface:
	# Native SAGE has two interfaces, one for univariate polynomials,
	# and a separate one for multivariate polynomials, which is annoying.

	R = P.parent()
	S = R.base_ring()

	if R.ngens() == 1:
		return S((list(P) + [0, 0])[1])
	
	else:
		return S(P.coefficient({v:1}))


def eliminate(L, bounds, integral=True):
	# Apply Lemma C.1.

	R = sum(L).parent()
	v = R.gens()[-1]

	zero = []
	pos = []
	neg = []

	for i in L:
		s, bounds = asymptotic_sign(coef(i, v), bounds)

		if s == 0:
			zero.append(i)
		elif s == 1:
			pos.append(i)
		elif s == -1:
			neg.append(i)

	for i in pos:
		for j in neg:
			iv = coef(i, v)
			jv = coef(j, v)

			if integral:
				zero.append(iv * j - jv * i + (iv - 1) * (jv + 1))
			else:
				zero.append(iv * j - jv * i)

	n = R.ngens()

	if n == 1:
		Rp = R.base_ring()

	else:
		Rp = PolynomialRing(R.base_ring(), R.gens()[:-1])
		D = {R.gens()[i] : Rp.gens()[i] for i in xrange(n - 1)}
		zero = [i.subs(D) for i in zero]
	
	return ([Rp(i) for i in zero], bounds)


def is_compact(L, bounds):
	# Determine if region satisfying inequalities L is compact when
	# all base variables are large, plus lower bounds on the base variables
	# which ensure this is the case and imply the inputted bounds
	# (second argument).

	R = sum(L).parent()
	variables = R.gens()
	n = len(variables)

	for M in subsets(L, n + 1):
		boundsM = bounds[:]
		cpt = True
		ssf = 0
		for i in xrange(n + 1):
			s, boundsM = asymptotic_sign(matrix([[coef(e, v) for v in variables] for e in (M[:i] + M[i+1:])]).det(), boundsM)
	
			if s in (0, ssf):
				cpt = False
				break

			else:
				ssf = s

		if cpt:
			return (True, boundsM)
	
	return (False, bounds)


def vertices(target, bounds):
	# Find vertices of region satisfying inequalities target when
	# all base variables are large, plus lower bounds on the base variables
	# which ensure this is the case and imply the inputted bounds
	# (second argument).

	R = sum(target).parent()
	variables = R.gens()
	n = len(variables)
	S = R.base_ring()

	verts = []

	for T in subsets(target, n):
		M = matrix([[coef(i, j) for j in variables] for i in T])
		V = vector([S(i.constant_coefficient()) for i in T])

		try:
			L = list(-M.inverse() * V)
		except ZeroDivisionError:
			continue

		actualvertex = True

		for j in target:
			s, p = asymptotic_sign(S(j.subs({variables[i]:L[i] for i in xrange(n)})), bounds)
			if s == -1:
				actualvertex = False
				bounds = p
				break
		
		if actualvertex:
			verts.append(L)

	return (verts, bounds)


def contains(L, p, bounds):
	# Check that p satisfies the inequalities L when all base variables
	# are large, plus find lower bounds on the base variables which ensure
	# this is the case and imply the inputted bounds (third argument).

	R = sum(L).parent()
	variables = R.gens()
	n = len(variables)
	S = R.base_ring()

	pD = {variables[i]:p[i] for i in xrange(n)}
	cont = True
	for i in L:
		s, bounds = asymptotic_sign(S(i.subs(pD)), bounds)

		if s == -1:
			cont = False
			break

	return (cont, bounds)


def asymptotic_covers(A, B, C, bounds):
	# Try to prove every assignment of integers satisfying C can be
	# completed to one satisfying either A or B when constant coefficients
	# are rounded up, when the base variables are all sufficiently large;
	# also find lower bounds on the base variables which ensure
	# this is the case and imply the inputted bounds (fourth argument).
	#
	# If there are no base variables (base ring is QQ), then instead try
	# to prove all but finitely many such assigments can be completed as
	# above; also return this finite list of those that cannot.

	Ra = sum(A).parent()
	Rb = sum(B).parent()
	Rc = sum(C).parent()
	S = Rc.base_ring()

	while Ra.ngens() != Rc.ngens():
		A, bounds = eliminate(A, bounds)
		Ra = sum(A).parent()

	while Rb.ngens() != Rc.ngens():
		B, bounds = eliminate(B, bounds)
		Rb = sum(B).parent()

	R = Rc
	variables = R.gens()
	n = len(variables)
	S = R.base_ring()

	# Ensure C is compact:
	ic, bounds = is_compact(C, bounds)
	if not ic:
		return (False, bounds)

	# We may now apply Lemma C.3 (plus Remark C.4):

	verts, bounds = vertices(C, bounds)

	fail = False
	inA = []
	inB = []
	for v in verts:
		try:
			via, tmp = contains(A, v, bounds)
			if via:
				bounds = tmp
		except NotImplementedError:
			via = False

		try:
			vib, tmp = contains(B, v, bounds)
			if vib:
				bounds = tmp
		except NotImplementedError:
			vib = False

		if not (via or vib):
			fail = True
			break
		elif not via:
			inB.append(v)
		elif not vib:
			inA.append(v)

	if not fail:
		St = S['t']
		t = St.gen()
		for v in inA:
			for w in inB:
				boundsvw = bounds[:]
				Dvw = {variables[j]: (v[j] + w[j])/2 for j in xrange(n)}
				vanishes = []
				for i in C:
					s, boundsvw = asymptotic_sign(S(i.subs(Dvw)), boundsvw)
					if s == 0:
						vanishes.append([coef(i, x) for x in variables])
				if matrix(vanishes).rank() >= n - 1:
					# vw is an edge of C.
					D = {variables[i]: t * v[i] + (1 - t) * w[i] for i in xrange(n)}
					inter = [St(i.subs(D)) for i in A + B]
					inter, bounds = eliminate(inter, bounds, integral=False)
					for i in inter:
						s, bounds = asymptotic_sign(i, bounds)
						if s == -1:
							fail = True
							break
				else:
					# vw is not an edge of C.
					bounds = boundsvw

				if fail:
					break
			if fail:
				break
	
	if not fail:
		return (True, bounds)
	
	if S != QQ:
		return (False, bounds)

	A = Polyhedron(ieqs=[[QQ(p.constant_coefficient())] + [coef(p, v) for v in variables] for p in A])
	B = Polyhedron(ieqs=[[QQ(p.constant_coefficient())] + [coef(p, v) for v in variables] for p in B])
	C = Polyhedron(ieqs=[[QQ(p.constant_coefficient())] + [coef(p, v) for v in variables] for p in C])
		
	bad = []
	lower, upper = C.bounding_box()
	for p in cartesian_product_iterator([xrange(int(ceil(lower[i])), int(floor(upper[i]) + 1)) for i in xrange(n)]):
		
		if C.contains(p):
			if not (A.contains(p) or B.contains(p)):
				bad.append(list(p))
	
	return (True, bad)


def covers(A, B, C, bounds):
	# Documented in beginning of Appendix B.

	ac, b = asymptotic_covers(A, B, C, bounds)

	if not ac:
		return False

	Ra = sum(A).parent()
	Rb = sum(B).parent()
	Rc = sum(C).parent()
	S = Rc.base_ring()

	if S == QQ:
		n = Rc.ngens()
		baseA = Ra.gens()[:n]
		variablesA = Ra.gens()[n:]
		baseB = Rb.gens()[:n]
		variablesB = Rb.gens()[n:]

		for p in b:
			Ap = [i.subs({baseA[i] : p[i] for i in xrange(n)}) for i in A]
			Bp = [i.subs({baseB[i] : p[i] for i in xrange(n)}) for i in B]

			Ap = Polyhedron(ieqs=[[QQ(ceil(QQ(i.constant_coefficient())))] + [coef(i, v) for v in variablesA] for i in Ap])
			Bp = Polyhedron(ieqs=[[QQ(ceil(QQ(i.constant_coefficient())))] + [coef(i, v) for v in variablesB] for i in Bp])
			
			if len(Ap.integral_points()) == len(Bp.integral_points()) == 0:
				return False

		return True

	n = S.ring().krull_dimension()

	if n == 1:
		for r in xrange(bounds[0], b[0]):
			Ar = [i.change_ring(Hom(S, QQ)(r)) for i in A]
			Br = [i.change_ring(Hom(S, QQ)(r)) for i in B]
			Cr = [i.change_ring(Hom(S, QQ)(r)) for i in C]

			if not covers(Ar, Br, Cr, []):
				return False

		return True
	
	elif n == 3:
		U = QQ['x'].fraction_field()
		x = U.gen()
		T = S.base_ring().fraction_field()
		r, k = T.gens()

		for t in xrange(bounds[0], b[0]):
			Bt = binomial(r + k, t)
			Ap = [i.change_ring(Hom(S, T)(Bt)).change_ring(Hom(T, U)([t, x])) for i in A]
			Bp = [i.change_ring(Hom(S, T)(Bt)).change_ring(Hom(T, U)([t, x])) for i in B]
			Cp = [i.change_ring(Hom(S, T)(Bt)).change_ring(Hom(T, U)([t, x])) for i in C]

			if not covers(Ap, Bp, Cp, [bounds[1]]):
				return False

		for t in xrange(bounds[1], b[1]):
			Bt = binomial(r + k, t)
			Ap = [i.change_ring(Hom(S, T)(Bt)).change_ring(Hom(T, U)([x, t])) for i in A]
			Bp = [i.change_ring(Hom(S, T)(Bt)).change_ring(Hom(T, U)([x, t])) for i in B]
			Cp = [i.change_ring(Hom(S, T)(Bt)).change_ring(Hom(T, U)([x, t])) for i in C]

			if not covers(Ap, Bp, Cp, [b[0]]):
				return False

		return True

	raise NotImplementedError
\end{lstlisting}

\bibliographystyle{amsplain.bst}
\bibliography{mrcbib}

\end{document}